\documentclass[11pt,a4paper]{article}
\usepackage{latexsym,amsfonts,amsmath,graphics,mathtools,amssymb,amsthm,booktabs}
\usepackage{epsfig}
\usepackage[notref,notcite]{showkeys}
\usepackage{enumitem}
\usepackage[usenames,dvipsnames,svgnames,table]{xcolor}
\usepackage[normalem]{ulem}
\usepackage{algorithmic}
\usepackage{hyperref}
\usepackage{mathrsfs}
\usepackage{cleveref}

\newcommand{\thed}{\mathscr{D}}
\newcommand{\theM}{\mathscr{E}}
\newcommand{\theUB}{\mathscr{U}}
\newcommand{\kor}{{\mathrm{kor}}}
\newcommand{\bsa}{\boldsymbol{a}}
\newcommand{\bsb}{\boldsymbol{b}}
\newcommand{\bsh}{\boldsymbol{h}}
\newcommand{\bsx}{\boldsymbol{x}}
\newcommand{\ZZ}{{\mathbb{Z}}}
\newcommand{\NN}{{\mathbb{N}}}

\DeclareSymbolFont{GreekLetters}{OML}{cmr}{m}{it} 
\DeclareSymbolFont{UpSfGreekLetters}{U}{cmss}{m}{n} 
\DeclareMathSymbol{\varrho}{\mathalpha}{GreekLetters}{"25}
\DeclareMathSymbol{\UpSfLambda}{\mathalpha}{UpSfGreekLetters}{"03}
\DeclareMathSymbol{\UpSfSigma}{\mathalpha}{UpSfGreekLetters}{"06}
\providecommand{\mathbold}{\boldsymbol}
\newcommand{\bvec}[1]{\mathbold{#1}}

\newlength{\overwdth}

\makeatletter
\newcommand*\bigcdot{\mathpalette\bigcdot@{.7}}
\newcommand*\bigcdot@[2]{\mathbin{\vcenter{\hbox{\scalebox{#2}{$\m@th#1\bullet$}}}}}
\makeatother

\def\abs#1{\ensuremath{\left \lvert #1 \right \rvert}}

\newcommand{\norm}[2][{}]{\ensuremath{\left \lVert #2 \right \rVert}_{#1}}

\newcommand{\ip}[3][{}]{\ensuremath{\left \langle #2, #3 \right \rangle_{#1}}}

\DeclareMathOperator{\comp}{comp}

\newcommand{\vzero}{\bvec{0}}

\newcommand{\vp}{\bvec{p}}

\newcommand{\hI}{\hat{I}}

\newcommand{\naturals}{\mathbb{N}}

\newcommand{\natzero}{\mathbb{N}_{0}}

\newcommand{\cb}{\mathcal{B}}

\newcommand{\cf}{\mathcal{F}}
\newcommand{\cg}{\mathcal{G}}
\newcommand{\ch}{\mathcal{H}}

\newcommand{\cp}{\mathcal{P}}

\DeclareMathOperator{\card}{card}

\DeclareMathOperator{\SOL}{SOL}
\DeclareMathOperator{\APP}{APP}
\renewcommand{\hI}{\mathcal{I}}

\allowdisplaybreaks

\newcommand{\vinfty}{\boldsymbol{\infty}}

\newtheorem{theorem}{Theorem}
\newtheorem{proposition}{Proposition}
\theoremstyle{definition}
\newtheorem{definition}{Definition}
\newtheorem{example}{Example}

\begin{document}

\title{A Unified Treatment of Tractability for \\ Approximation Problems Defined on Hilbert Spaces}

\author{Onyekachi Emenike, Fred J. Hickernell, Peter Kritzer}

\date{\today}
\maketitle

\begin{abstract}
     A large literature specifies conditions under which the information complexity for a sequence of numerical problems defined for dimensions $1, 2, \ldots$ grows at a moderate rate, i.e., the sequence of problems is \emph{tractable}.  Here, we focus on the situation where the space of available information consists of all linear functionals, and the problems are defined as linear operator mappings between Hilbert spaces.  We unify the proofs of known tractability results and generalize a number of existing results.  These generalizations are expressed as five theorems that provide equivalent conditions for (strong) tractability in terms of sums of functions of the singular values of the solution operators.
\end{abstract}


\section{Introduction}

\noindent The information complexity of a problem is the number of function data required to solve the problem within the desired error tolerance.  A problem is tractable if the information complexity does not grow too large as the dimension of the problem increases or the error threshold decreases.  There are a number of tractability results for numerical problems, beginning in the 1990s with \cite{W94a,W94b}, spawning numerous articles since then, and filling a series of three impressive volumes \cite{NovWoz08a,NovWoz10a,NovWoz12a}.  Research on tractability continues at a vigorous pace. 

The authors have identified common themes in the proofs of necessary and sufficient conditions for tractability that appear in the literature.  By abstracting these themes, this article aims to unify and generalize many existing tractability results for approximation problems where all possible linear functional data are available. This is done in a series of five theorems, which provide necessary and sufficient conditions for (strong) tractability.

Let $\{\cf_d\}_{d \in \naturals}$ and $\{\cg_d\}_{d \in \naturals}$ be sequences of Hilbert spaces, and let $\{\SOL_d : \cf_d \to \cg_d\}_{d \in \naturals}$ be a sequence of linear solution operators  with adjoints $\SOL_d^*$ such that $\SOL_d^*\SOL_d : \cf_d \to \cf_d$ has eigenvalues and orthonormal eigenvectors
\begin{equation*} 
 \lambda_{1,d}^2 \ge \lambda_{2,d}^2 \ge \cdots, \qquad u_{1,d}, u_{2,d}, \ldots, \qquad d \in \naturals.
\end{equation*}
Here, $\naturals$ denotes the set of positive integers.  This means that the singular values of the operator $\SOL_d$ are the non-negative $\lambda_{i,d}$. We remark that in the literature on Information-Based Complexity the singular values of $\SOL_d$ are often denoted by $\sqrt{\lambda_{i,d}}$ instead of $\lambda_{i,d}$. However, since we assume non-negativity of the $\lambda_{i,d}$, the notation used here is equivalent.

We make the technical assumption that $\SOL_d$ is a compact operator. This implies (see, e.g., \cite{NovWoz08a}) that $\lambda_{i,d}$ converges to zero as $i$ tends to infinity for every $d$, which in turn implies that the problem is solvable by suitable algorithms. To avoid trivial cases, we assume that there are an infinite number of positive singular values for every $d$.

The goal is to find a sequence of approximate solution operators, $\{\APP_d:\cb_d \times (0,\infty) \to \cg_d\}_{d \in \naturals}$, defined on the unit ball of functions in $\cf_d$, denoted $\cb_d$, which  satisfy an absolute error criterion:
\begin{equation}
    \label{eq:errcrit}
    \norm[\cg_d]{\SOL_d(f) - \APP_d(f,\varepsilon)} \le \varepsilon \qquad \forall f \in \cb_d, \ \varepsilon \in (0,\infty).\footnote{We note that the radius of the $\cb_d$ could be taken to be an arbitrary positive value, $R$, in which case the problem becomes equivalent to the original one with the tolerance $\varepsilon$ replaced by $\varepsilon/R$ since the problem and the optimal approximate solution are both linear (see below). We allow the error tolerance, $\varepsilon$, to be arbitrarily large because if $\SOL_d$ has large singular values, a large error tolerance may be a reasonable expectation.}
\end{equation}
Here, $\APP_d(f,\varepsilon)$ is allowed to depend on arbitrary linear functionals.  Note that we also allow adaptive information, i.e., the linear functionals may sequentially depend on each other.

The information complexity of the problem $\SOL_d$ is given by the function 
\[
\comp:(0,\infty) \times \naturals \to \natzero, 
\]
where $\comp(\varepsilon,d)$ is the number of linear functionals required by the best admissible algorithm to satisfy the error criterion \eqref{eq:errcrit}, and  $\naturals_0$ denotes the set of non-negative integers.
By definition, the information complexity is non-decreasing as $\varepsilon$ tends to zero.  We also expect it to be non-decreasing as $d$ tends to infinity since $d$ is typically the number of variables for the functions in $\cf_d$.

For this case where arbitrary linear functionals are allowed, the optimal approximate solution operator is known to be
\[
\APP_d(f,\varepsilon) = \sum_{i=1}^n \SOL_d(u_{i,d}) \ip[\cf_d]{f}{u_i},
\]
and the information complexity of our linear problem is
\begin{equation}\label{eq:comp_def}
\comp(\varepsilon, d) = \min \{n \in \natzero : \lambda_{n+1, d} \le \varepsilon\},
\end{equation}
see \cite{NovWoz08a}.

Strictly speaking, we are considering the \textit{absolute} error criterion in this paper.
Alternatively, one could also consider the \textit{normalized} error criterion, where the error in \eqref{eq:errcrit} is divided by the so-called \textit{initial error}, which is the error without sampling the function, or, equivalently, the operator norm of $\SOL_d$ (see \cite{NovWoz08a} for details). For the sake of brevity, we restrict ourselves to the absolute error criterion here, but presumably, analogous results hold for the normalized setting (all criteria would then be normalized by the first singular value $\lambda_{1,d}$).

The question of tractability is one of determining how fast the information complexity increases as $\varepsilon$ tends to zero and/or as $d$ tends to infinity.  Although \eqref{eq:comp_def} has a relatively simple form, its dependence on the ordering of the singular values means that it is not obvious what the behavior of $\comp(\varepsilon, d)$ is as $\varepsilon^{-1}$ and/or $d$ tend to infinity.

We want to bound the information complexity in terms of a simple function, $T$, of $\varepsilon^{-1}$, $d$, and some parameter $\vp$, and then identify conditions on the singular values that are equivalent to the fact that this bound holds, and are easier to verify than showing the bound directly. By ``simple'' we mean that $T$ takes a form such that the conditions in the theorems of this paper can be checked effectively; however, we remark that this is also related to the singular values $\lambda_{i,d}$ of the solution operator. Depending on those, a function $T$ may be useful for one particular problem setting but not for another. In Table \ref{tab:sampleT}, we present various common notions of tractability studied in the literature, characterized by $T$. 
\begin{table}[hbt!]
    \caption{Common notions of tractability}
{\small
\begin{equation*}
	\begin{array}{r@{\quad}c@{\quad}cc}
		\text{Tractability type} & T(\varepsilon^{-1},d,\vp)
		\\
		\toprule
		\text{algebraic strong polynomial} & \max\{1,\varepsilon^{-p}\} \\
		\text{algebraic polynomial} & \max\{1,\varepsilon^{-p}\}d^{q} \\
		\text{exponential strong polynomial} &  [\max\{1,\log(1 + \varepsilon^{-1}\}]^p \\
		\text{exponential polynomial} &
		[\max\{1,\log(1 + \varepsilon^{-1})\}]^p  d^{q} \\
        \text{algebraic quasi-polynomial} &
        \exp\{p(1+\log(\max\{1, \varepsilon^{-1}\}))(1+\log(d))\} \\
        \text{exponential quasi-polynomial} &
        \exp\{p(1+\log(\max\{1,\log(1+\varepsilon^{-1})\}))(1+\log(d))\} 
	\end{array}
\end{equation*}}
\label{tab:sampleT}
\end{table}

We note that generalized tractability specified by a function $T$ has been considered before in the literature. In particular, there are several papers by Gnewuch and Wo\'{z}niakowski (see \cite{GW06,GW08,GW09,GneWoz11a}), and also \cite[Chapter 8]{NovWoz08a} is devoted to this topic. To our understanding, however, these references mainly focus on the special case of tensor product problems, a restriction we do not make.  Moreover, the conditions for tractability derived there focus mostly on relating the properties of the function $T$ to the different varieties of tractability, and not in terms of sums of (functions of) the singular values of a problem. Hence, the present paper adds to what is presently known on this subject. 

In this article, we investigate a \emph{generalized tractability function}, denoted by $T$:
\begin{equation} \label{eq:Tspec}
    T :(0,\infty) \times \mathbb{N} \times [0,\infty)^s \rightarrow (0,\infty).
\end{equation}
The conditions we require $T$ to satisfy are described in  Section \ref{sec:not}, however, the basic idea is that we define our approximation problem to be tractable if $\comp(\varepsilon, d) \le  C_{\vp}\, T(\varepsilon^{-1},d,\vp)$ for some constant\footnote{By ``constant'' we mean that the term is independent of the \emph{variables} $\varepsilon$ and $d$, but may depend on the \emph{parameter} $\vp$.}, $C_{\vp}$, depending only on the parameter $\vp$. For instance, when considering (algebraic) polynomial tractability, the parameter $\vp$ represents the exponents $(p,q)$ of $\varepsilon$ and $d$, as illustrated in Table \ref{tab:sampleT}.  In practice, we expect $T(\varepsilon^{-1},d,\vp)$ to increase sub-exponentially with increasing $\varepsilon^{-1}$ and/or $d$, however, most of our theorems do not require such an assumption.

The parameter $\vp$ is an $s$-dimensional vector with $s \geq 1$. We write $\vzero$ to denote a vector with all components equal to zero, and $\vinfty$ to denote a vector with all components equal to $\infty$.
Furthermore, for two vectors $\vp=(p_1,\ldots,p_s)$ and $ \vp'=(p_1',\ldots,p_s')$ in $[0,\infty)^s$, 
we write $\vp\ge \vp'$ if $p_j\ge p'_j$ for all $j\in \{1,\ldots,s\}$. Furthermore, $\vp > \vp'$ 
if $p_j> p'_j$ for all $j\in \{1,\ldots,s\}$. The expression $\vp\in (\vp',\vp'')$ is to be interpreted as 
$\vp' < \vp < \vp''$ in that sense, and analogously for half-closed or closed intervals.

\Cref{thm_main_strong_tract2,thm_main_tract2,thm_main_rest_strong_tract,thm_main_rest_tract} below all take the following form:  for a given $\thed \in \{1,d\}$ and $\Omega \subseteq (0,\infty) \times \mathbb{N}$,
\begin{multline} \label{eq:metatheorem}
  \exists \,C_{\vp} > 0 \text{ with }
    \comp(\varepsilon, d) \le  C_{\vp}\, T(\varepsilon^{-1},\thed,\vp) \quad \forall (\varepsilon^{-1},d) \in \Omega \\
    \iff   \exists L_{\vp} >0 \text{ with }
    \sup_{d \in \naturals}
     \sum_{i = \lceil L_{\vp}\, T(0,\thed,\vp) \rceil}^{\theUB(d)} \frac{1}{T(\lambda^{-1}_{i,d},\thed,\vp)}< \infty, \\
     \text{where } \theM(d) := \inf \{\varepsilon : (\varepsilon^{-1},d) \in \Omega \}, \
     \theUB(d) := \begin{cases}
        \min\{n \in \natzero : \lambda_{n+1,d} \le \theM(d) \}, &   \theM(d) > 0, \\
        \infty, & \theM(d) = 0,
        \end{cases}
\end{multline}
where $T(0,\thed,\vp)$ is defined as in \eqref{eq:T0} below. These theorems say essentially that \emph{(strong) tractability is equivalent to summability conditions on the singular values.} A more stringent form of tractability resulting from a smaller $T$ is equivalent to a more stringent summability condition since the terms being summed become larger.

Although some of \Cref{thm_main_strong_tract2,thm_main_tract2,thm_main_rest_strong_tract,thm_main_rest_tract} are special cases of others, we prove each separately to allow the reader to become familiar with the arguments used as additional layers of complexity are added. (This may be, e.g., additional
dependence on $d$ when going from strong tractability to tractability in Subsection~\ref{sec:tractability}, or a restriction of the domain of $\varepsilon$ for which we define tractability in Section 
~\ref{sec:restricted}).  \Cref{thm:subhT} has a somewhat different form than \eqref{eq:metatheorem} above.  

Although in many cases $\Omega$ is chosen to be all of $(0,\infty) \times \mathbb{N}$, there are situations where it makes sense to choose $\Omega$ as a proper subset of all possible values of $\varepsilon^{-1}$ and $d$.  For some problems with unbounded $d$---as in finance---there may be no need for arbitrarily small error tolerances, so $\Omega = (0,\varepsilon_{\min}^{-1}] \times \mathbb{N}$ may be appropriate.  For other problems with a moderate upper bound on $d$, $\Omega = (0,\infty) \times \{1, \ldots, d_{\max}\}$ may be appropriate.

The remainder of this paper is organized as follows. In Section \ref{sec:spt}, we give the definitions of the considered tractability notions, in particular (strong) tractability. We consider strong tractability when there is an upper bound on the information complexity that is independent of the dimension, $d$. We then turn to tractability when the bound on the information complexity may depend on both $d$ and $\varepsilon^{-1}$. In Section \ref{sec:examples} we provide examples of tractability functions for both the algebraic and exponential cases. We also consider other notions of tractability such as quasi-polynomial tractability, and introduce other notions of tractability. A generalized notion of weak tractability is given in Section \ref{sec:subh}. In Section \ref{sec:restricted}, we consider tractability on a restricted domain, $\Omega$. There are two appendices to the paper. Appendix 1 contains an example that shows how we can make use of one of the findings in this paper in a concrete example. Appendix 2 contains the technical proof of Proposition \ref{prop:equiv_conditions}.

\section{(Strong) Tractability}\label{sec:spt}

\subsection{Notation and fundamental definitions} \label{sec:not}
Let $T$ be a function given in \eqref{eq:Tspec}.
A tractability function $T$ is a function of a simple form that provides an upper bound on the information complexity of a problem.

\begin{definition}
    A problem is defined as $T$-tractable with parameter $\vp$ iff there exists a positive constant $C_{\vp}$, which is independent of $\varepsilon$ and $d$, such that
\begin{equation} \label{eq:tract_def}
	\comp(\varepsilon,d) \le C_{\vp}\, T(\varepsilon^{-1},d,\vp) \qquad \forall \varepsilon >0, \ d \in \mathbb{N}.
\end{equation}
A problem is \emph{strongly}
$T$-tractable with parameter $\vp$ iff the information complexity is independent of the dimension of the problem, that is, there exists a positive constant $C_{\vp}$, again independent of $\varepsilon$ and $d$, such that
\begin{equation} \label{eq:strong_tract_def}
	\comp(\varepsilon,d) \le C_{\vp}\, T(\varepsilon^{-1},1,\vp) \qquad \forall \varepsilon >0 , \ d \in \mathbb{N}.
\end{equation}
We allow $\vp$ to be a scalar or vector as the situation dictates.
\end{definition}

For this definition of tractability to make sense we assume that
\begin{subequations} \label{eq:Tconditions}
\begin{equation} \label{eq:Tnondecreasing}
	T \text{ is non-decreasing in all variables,}
\end{equation}
which implies that the problem is expected to be no easier by decreasing $\varepsilon$, the tolerance, or increasing $d$, the dimension\footnote{This is why we make $T$ a function of $\varepsilon^{-1}$ rather than of $\varepsilon$.}. Furthermore, increasing $\vp$ allows for a possibly looser bound on the information complexity. Since there are an infinite number of positive singular values, it also makes sense to assume that
\begin{equation}\label{eq:Tlim}
	\lim_{\varepsilon \to 0} T(\varepsilon^{-1},d,\vp) = \infty \qquad \forall d \in \naturals, \ \vp \in [0,\infty)^s.
\end{equation}
Since $T(\cdot,d,\vp)$ is non-decreasing, we may define the following limit, which we assume will always be positive:
\begin{equation}\label{eq:T0}	T(0,d,\vp):=\lim_{\varepsilon\to\infty}T(\varepsilon^{-1},d,\vp) = \inf_\varepsilon T(\varepsilon^{-1},d,\vp) \ge T(0,1,\vzero) > 0.
\end{equation}
\end{subequations}
Finally, there is a technical assumption required for our analysis in \Cref{sec:spt,sec:restricted}.  There exists a $K_{\vp,\tau}$ depending on $\vp$ and $\tau$, but  independent of $\varepsilon$ and $d$, such that
\begin{equation} \label{eq:ptauassume}
	[T(\varepsilon^{-1},d,\vp)]^\tau \le K_{\vp,\tau} T(\varepsilon^{-1},d,\tau \vp)   \quad \forall \varepsilon \in (0,\infty), \ d \in \naturals, \ \vp\in[0,\infty)^s, \ \tau\in [1,\infty).
\end{equation}
We remark that the latter assumption is quite natural, and is, for example fulfilled for all tractability types in Table~\ref{tab:commonT}.

Certain common notions of tractability covered by the general situation described above are provided in Table \ref{tab:commonT}.  Here, $\vp = (p,q)$ in some cases, and $\vp$ is a scalar in other cases. A general reference
for results on algebraic tractability is \cite{NovWoz08a}.  For exponential tractability we refer to \cite{HicKriWoz19a} and \cite{KriWoz19a}.

\begin{table}
    \caption{Common forms of the tractability function, $T$.}
{\small
\begin{equation*}
	\begin{array}{r@{\quad}c@{\quad}cc}
		\text{Tractability type} & T(\varepsilon^{-1},d,\vp)
		& T(0,d,\vp)\\
		\toprule
		\text{algebraic polynomial} & \max\{1,\varepsilon^{-p}\}d^{q} & d^{q}\\
		\text{exponential polynomial} &
		[\max\{1,\log(1 + \varepsilon^{-1})\}]^p  d^{q} & d^{q} \\
        \text{algebraic quasi-polynomial} &
        \exp\{p(1+\log(\max\{1, \varepsilon^{-1}\}))(1+\log(d))\} &
        \exp\{p(1+\log(d))\}\\
        \text{exponential quasi-polynomial} &
        \exp\{p(1+\log(\max\{1,\log(1+\varepsilon^{-1})\}))(1+\log(d))\} &
        \exp\{p(1+\log(d))\}
	\end{array}
\end{equation*}}
\label{tab:commonT}
\end{table}

If \eqref{eq:tract_def} holds for some $\vp$, it clearly holds for larger $\vp$.  We are often interested in the optimal or smallest $\vp$ for which \eqref{eq:tract_def} holds.  Define the closures of the sets of parameters for which our (strong) tractability conditions hold:
\begin{equation*}
	\cp_{\textup{trct}} : = \{\vp^* : \eqref{eq:tract_def} \text{ holds }\forall \vp \in (\vp^*,\boldsymbol{\infty})\}, \qquad
	\cp_{\textup{strct}} : = \{\vp^* : \eqref{eq:strong_tract_def} \text{ holds }\forall \vp \in (\vp^*,\boldsymbol{\infty})\}.
\end{equation*}
If $\vp^* \in \cp_{\textup{(s)trct}}$, (strong) tractability may not hold for $\vp = \vp^*$, but it must hold for any $\vp$ whose components are all greater than the corresponding components of $\vp^*$.

\begin{definition}
    The  set of optimal parameters is defined as all of those parameters satisfying the (strong) tractability conditions that are not greater than or equal to others:
\begin{gather}
	\label{eq:Poptdef}
	\mathcal{P}_{\textup{opt}} : = \{\vp^* \in \cp_{\textup{trct}} :  \vp^* \notin [\vp,\boldsymbol{\infty}) \ \forall \vp \in  \cp_{\textup{trct}} \setminus \{\vp^*\} \}, \\
	\label{eq:Psoptdef}
	\mathcal{P}_{\textup{sopt}} : = \{\vp^* \in \cp_{\textup{strct}} :  \vp^* \notin [\vp,\boldsymbol{\infty}) \ \forall \vp \in  \cp_{\textup{strct}} \setminus \{\vp^*\} \}.
\end{gather}
\end{definition}

In the sections below we prove necessary and sufficient conditions on tractability as generally defined in \eqref{eq:tract_def} and \eqref{eq:strong_tract_def}.  These conditions involve the boundedness of sums defined in terms of $(T(\lambda_{i,d}^{-1},\cdot, \cdot))_{i=1}^\infty$.  In practice, it may be easier to verify whether or not these conditions hold than to verify \eqref{eq:tract_def} and \eqref{eq:strong_tract_def} directly.

\subsection{Strong tractability}

We first consider the simpler case when the information complexity is essentially independent of the dimension, $d$.  The proof also introduces the line of argument used for the case where there is $d$-dependence.

\begin{theorem}\label{thm_main_strong_tract2}
Let $T$ be a tractability function as specified in \eqref{eq:Tspec} and satisfying \eqref{eq:Tconditions} and \eqref{eq:ptauassume}.  A problem is strongly $T$-tractable iff there exists $\vp \in [\vzero, \boldsymbol{\infty})$ and an integer $L_{\vp} > 0$ such that
\begin{equation} \label{eq:strong_tractiff3}
     S_{\vp}:=\sup_{d \in \naturals} \sum_{i = L_{\vp}}^\infty \frac{1}{T(\lambda_{i,d}^{-1},1,\vp)} < \infty.
\end{equation}
If \eqref{eq:strong_tractiff3} holds for some $\vp$, let  $\widetilde{\cp}_{\textup{strct}} : = \{\vp^* : \eqref{eq:strong_tractiff3} \text{ holds }\forall \vp \in (\vp^*,\boldsymbol{\infty})\}$.  Then $\cp_{\textup{strct}} = \widetilde{\cp}_{\textup{strct}}$, and the set of optimal strong tractability parameters is
\[
	\cp_{\textup{sopt}} =
	\{\vp^* \in \widetilde{\cp}_{\textup{strct}} :  \vp^* \notin [\widetilde{\vp},\boldsymbol{\infty}) \ \forall \widetilde{\vp}\in  \widetilde{\cp}_{\textup{strct}} \setminus \{\vp^*\} \}.
\]
\end{theorem}

\begin{proof}
\textbf{Sufficient condition:}\newline
We make the first part of the argument in some generality so that it can be reused in the proof of Theorem \ref{thm_main_tract2} for tractability.  Fix any $\varepsilon > 0$ and any $d \in \naturals$, and let $\thed \in \naturals$ be arbitrary. 
Since the $\lambda_{i,d}$ are non-increasing in $i$, it follows that the $T(\lambda_{i,d}^{-1},\thed,\vp)$ are non-decreasing in $i$. In particular, for  any $N\in\naturals$ with $n\ge N$, we have
\begin{align*}
    \lambda_{n+1,d} \le \lambda_{n,d} \le \cdots \le \lambda_{N,d}
    & \implies \frac{1}{T(\lambda_{n+1, d}^{-1},\thed,\vp)} \le \frac{1}{T(\lambda_{n, d}^{-1},\thed,\vp)} \le \cdots \le \frac{1}{T(\lambda_{N, d}^{-1},\thed,\vp)} \\
     \implies \frac{1}{T(\lambda_{n+1, d}^{-1},\thed,\vp) }
    \le &\frac{1}{n-N+1} \sum_{i=N}^n  \frac{1}{T(\lambda_{i, d}^{-1},\thed,\vp) }
    \le \frac{1}{n-N+1} \sum_{i=N}^\infty  \frac{1}{T(\lambda_{i, d}^{-1},\thed,\vp)},
\end{align*}
where the infinite sum is guaranteed to exist by the condition \eqref{eq:strong_tractiff3}.
Thus, we can conclude from the previous line that
\begin{multline} \label{eq:implicate_a}
    n - N +1 \ge T(\varepsilon^{-1},\thed,\vp) \sum_{i=N}^\infty \frac{1}{T(\lambda_{i, d}^{-1},\thed,\vp)} \\
   \implies   \frac{1}{T(\lambda_{n+1, d}^{-1},\thed,\vp)} \le
   \frac{1}{n-N+1} \sum_{i=N}^\infty \frac{1}{T(\lambda_{i, d}^{-1},\thed,\vp) } \le \frac{1}{T(\varepsilon^{-1},\thed,\vp)} \\ \forall \thed, N \in \naturals\ \mbox{and}\ n\ge N.
\end{multline}

Moreover, given that $T$ is increasing in its arguments, we have the following equivalent expression for the information complexity via \eqref{eq:comp_def}:
\begin{align} 
	\nonumber
	\comp(\varepsilon, d) & = \min \{n \in \natzero : \lambda_{n+1, d} \le \varepsilon\} \\
	\nonumber
	& = \min \biggl\{n \in \natzero : \frac{1}{T(\lambda_{n+1,d}^{-1},\thed,\vp)}\le \frac{1}{T(\varepsilon^{-1},\thed,\vp) }\biggr\} \qquad  \forall \thed\in \naturals\\
	\nonumber
	& \le  \min \biggl\{n \in \natzero : n \ge N - 1 + T(\varepsilon^{-1},\thed,\vp) \sum_{i=N}^\infty \frac{1}{T(\lambda_{i, d}^{-1},\thed,\vp)} \biggr\} \;  \forall  \thed, N \in \naturals, \quad \text{by \eqref{eq:implicate_a}} \\
	& \le  T(\varepsilon^{-1},\thed,\vp) \Biggl [\frac{N}{T(\varepsilon^{-1},\thed,\vp)}  +  \sum_{i=N}^\infty \frac{1}{T(\lambda_{i, d}^{-1},\thed,\vp)} \Biggr] \quad \forall  \thed, N \in \naturals. \label{eq:compUP}
\end{align}

Now we take this upper bound further, specializing to the case of $\thed=1$ and $N=L_{\vp}$:
\begin{align*}
       \comp(\varepsilon,d)
       & \le T(\varepsilon^{-1},1, \vp) \Biggl [ \frac{L_{\vp}}{T(\varepsilon^{-1},1, \vp)}  + \underbrace{\sum_{i=L_{\vp}}^\infty \frac{1}{T(\lambda_{i, d}^{-1},1,\vp)}}_{\le S_{\vp} <\infty\ \text{by \eqref{eq:strong_tractiff3}} }
        \Biggr] \\
       & \le T(\varepsilon^{-1},1,\vp) \underbrace{\left[\frac{L_{\vp}}{T(0,1, \vp)} + S_{\vp} \right]}_{=:C_{\vp}}
       \qquad \text{by \eqref{eq:Tnondecreasing}}\\
       & =  C_{\vp}T(\varepsilon^{-1},1,\vp).
\end{align*}
This means that we have strong $T$-tractability via \eqref{eq:strong_tract_def}, and verifies the sufficiency of \eqref{eq:strong_tractiff3}.

\bigskip
\noindent \textbf{Necessary condition:} \\
Suppose that we have strong
$T$-tractability as defined in \eqref{eq:strong_tract_def}. That is, for some $\vp \ge \vzero$ there exists a positive constant $C_{\vp}$ such that
\[
\comp(\varepsilon,d)\le C_{\vp}\, T(\varepsilon^{-1},1,\vp)
\qquad \forall \varepsilon > 0, \ d \in \mathbb{N}.
\]
Since the sequence of singular values $\lambda_{1,d}, \lambda_{2,d}, \ldots $ is non-increasing, we have
\begin{equation}\label{eq:lambda_K_strong3}
\lambda_{\lfloor C_{\vp}\, T(\varepsilon^{-1},1,\vp)\rfloor +1,d}\le \varepsilon \qquad \forall \varepsilon > 0, \ d \in \mathbb{N}.
\end{equation}

For all positive $\varepsilon$, define
\begin{equation}\label{eq:iepsp}
i(\varepsilon,\vp):= \lfloor C_{\vp}\, T(\varepsilon^{-1},1,\vp)\rfloor +1, \quad
i(\infty,\vp) =\lfloor C_{\vp}\, T(0,1,\vp)\rfloor +1.
\end{equation}
Thus, it follows by \eqref{eq:lambda_K_strong3} that $\lambda_{i(\varepsilon,\vp),d} \le \varepsilon$.
Note furthermore that we always have
\[
i(\varepsilon,\vp)\le C_{\vp}\, T(\varepsilon^{-1},1,\vp)+1 \le C_{\vp} T(\lambda_{i(\varepsilon,\vp),d}^{-1},1,\vp)+1 \qquad \forall \varepsilon > 0, \ d \in \mathbb{N},
\]
since
$T(\cdot,1, \vp)$ is non-decreasing.

For $\varepsilon$ taking on all positive values, $i(\varepsilon,\vp)$ takes on, by \eqref{eq:Tlim}, all values greater than or equal to $i(\infty,\vp)$, so
\[
i\le C_{\vp} T(\lambda_{i,d}^{-1},1,\vp)+1 \qquad \forall i \ge i(\infty,\vp), \ d \in \naturals.
\]
This implies via our technical assumption \eqref{eq:ptauassume} that, for any $\tau>1$,
\begin{gather*}
 K_{\vp,\tau}\,T (\lambda_{i,d}^{-1},1,\tau \vp) \ge
 [T(\lambda_{i,d}^{-1},1, \vp)]^\tau
 \ge
  \left[\frac{(i-1)}{C_{\vp}}\right]^\tau \qquad \forall i \ge i(\infty,\vp), \ d \in \naturals, \\
 \frac{1}{T (\lambda_{i,d}^{-1},1,\tau \vp)} \le
\frac{K_{\vp,\tau}\, C_{\vp}^\tau}{(i-1)^\tau}\qquad \forall i \ge \max\{i(\infty,\vp),2\}, \ d \in \naturals .
\end{gather*}
Summing both sides of this last inequality over $i$ from $i(\infty,\vp) +1 \ge 2$ to $\infty$ yields
\begin{equation*}
\sup_{d\in\naturals} \sum_{i=i(\infty,\vp) +1 }^\infty \frac{1}{T (\lambda_{i,d}^{-1},1, \tau \vp)}
 \le  K_{\vp,\tau}\, C_{\vp}^\tau
\sum_{i=2}^\infty \frac{1}{(i-1)^\tau} \\
 \le K_{\vp,\tau}\, C_{\vp}^\tau
\zeta (\tau)  < \infty,
\end{equation*}
where $\zeta$ denotes the Riemann zeta function.
This yields \eqref{eq:strong_tractiff3} with $\vp$ replaced by $\vp' = \tau \vp$ and $L_{\vp'} = i(\infty,\vp) +1$, and so we see the necessity of $\eqref{eq:strong_tractiff3}$.

\bigskip
\noindent \textbf{Optimality:} \\
To complete the proof, we must show that $\cp_{\textup{strct}} = \widetilde{\cp}_{\textup{strct}}$.  Then, the expression for  $\cp_{\textup{sopt}}$ automatically follows from its definition in \eqref{eq:Psoptdef}.

First we show  that if $\vp^* \in \cp_{\textup{strct}}$, then $\vp^* \in \widetilde{\cp}_{\textup{strct}}$.  If $\vp^* \in \cp_{\textup{strct}}$, then \eqref{eq:strong_tract_def} must hold for all $\vp' \in (\vp^*,\vinfty)$.  For any $\vp' \in (\vp^*,\vinfty)$, we may choose $\vp$ and $\tau$ such that $\vp^* < \vp < \vp' = \tau \vp$. Since $\vp > \vp^*$, \eqref{eq:strong_tract_def} also holds for $\vp$, and it follows that \eqref{eq:strong_tractiff3} must hold for $\vp'$ by the proof of necessity above.  Thus, $\vp^* \in \widetilde{\cp}_{\textup{strct}}$.

Next we show that if $\vp^* \in \widetilde{\cp}_{\textup{strct}}$, then $\vp^* \in \cp_{\textup{strct}}$.  If $\vp^* \in \widetilde{\cp}_{\textup{strct}}$, then \eqref{eq:strong_tractiff3}  holds for all $\vp \in (\vp^*,\vinfty)$.  By the argument to prove the sufficient condition above, it follows that \eqref{eq:strong_tract_def} must also hold for all $\vp \in (\vp^*,\vinfty)$. Thus, $\vp^* \in \cp_{\textup{strct}}$.

\bigskip

\noindent This concludes the proof of  Theorem \ref{thm_main_strong_tract2}. 

\end{proof}

\subsection{Tractability} \label{sec:tractability}

The argument proving equivalent conditions for tractability is similar to, but somewhat more involved than the proof of Theorem \ref{thm_main_strong_tract2}.  The lower limit on the sum is somewhat more complicated as well.

\begin{theorem}\label{thm_main_tract2}
Let $T$ be a tractability function as specified in \eqref{eq:Tspec} and satisfying \eqref{eq:Tconditions} and \eqref{eq:ptauassume}.  A problem is $T$-tractable iff there exists $\vp\ge \vzero$ and a positive constant $L_{\vp}$ such that
\begin{equation} \label{eq:tractiff4}
     S_{\vp}:=\sup_{d \in \naturals}
     \sum_{i = \lceil L_{\vp}\, T(0,d,\vp) \rceil}^\infty \frac{1}{T(\lambda^{-1}_{i,d},d,\vp)}< \infty.
\end{equation}

If \eqref{eq:tractiff4} holds for some $\vp$, let $\widetilde{\cp}_{\textup{trct}} : = \{\vp^* : \eqref{eq:tractiff4} \text{ holds }\forall \vp \in (\vp^*,\boldsymbol{\infty}) \}$.
Then $\cp_{\textup{trct}} = \widetilde{\cp}_{\textup{trct}}$, and the set of optimal  tractability parameters is
\[
\cp_{\textup{opt}} =
\{\vp^* \in \widetilde{\cp}_{\textup{trct}} :  \vp^* \notin [\widetilde{\vp},\boldsymbol{\infty}) \ \forall \widetilde{\vp}\in  \widetilde{\cp}_{\textup{trct}} \setminus \{\vp^*\} \}.
\]

\end{theorem}

Comparing the equivalent conditions for strong tractability in \eqref{eq:strong_tractiff3} and tractability in \eqref{eq:tractiff4}, they are seen to be nearly the same.  The difference lies in the lower summation index, which is allowed to depend on $d$ for tractability, and the second argument of $T$ in the sum, which is $1$ for strong tractability, and $d$ in general.  Both of these can allow $S_{\vp}$ to be finite in Theorem \ref{thm_main_tract2}, when it may be infinite in Theorem \ref{thm_main_strong_tract2}.

\begin{proof}
    \textbf{Sufficient condition:}\\
Suppose that \eqref{eq:tractiff4} holds for some $\vp \ge \boldsymbol{0}$.
By the argument in the proof of the sufficient condition for Theorem \ref{thm_main_strong_tract2}, it follows from \eqref{eq:compUP} with $\thed=d$ and $N = \lceil L_{\vp}\, T(0,d,\vp)\rceil$ that
\begin{align*}
       \comp(\varepsilon,d)
       & \le T(\varepsilon^{-1},d,\vp) \Biggl [
       \frac{\lceil L_{\vp}\, T(0,d,\vp)\rceil}{T(\varepsilon^{-1},d,\vp)}  +
       \underbrace{\sum_{i=\lceil L_{\vp}\, T(0,d,\vp)\rceil}^\infty \frac{1}{T(\lambda_{i, d}^{-1},d,\vp)}}_{\le S_{\vp} \text{ by \eqref{eq:tractiff4}}}
       \Biggr ]
        \quad \forall \varepsilon > 0, \ d  \in \naturals\\
       & \le T(\varepsilon^{-1},d,\vp)\left[\frac{L_{\vp} T(0,d,\vp)+1}{T(\varepsilon^{-1},d, \vp)} +  S_{\vp} \right]
        \qquad \forall \varepsilon > 0, \ d  \in \naturals\\
       & \le T(\varepsilon^{-1},d,\vp)\left[\frac{L_{\vp} T(0,d,\vp)}{T(\varepsilon^{-1},d, \vp)}
       + \frac{1}{T(0,1, \vp)} + S_{\vp}\right]
       \qquad \text{by \eqref{eq:Tnondecreasing}}   \\
       & \le T(\varepsilon^{-1},d,\vp)
       \underbrace{\left [  L_{\vp} + \frac{1}{T(0,1, \vp)} + S_{\vp} \right]}_{=:C_{\vp}}
       \qquad \text{by \eqref{eq:Tnondecreasing}} .
\end{align*}
It follows that we have $T$-tractability, which shows sufficiency of \eqref{eq:tractiff4}.

\bigskip

\noindent \textbf{Necessary condition:}\\
Suppose that we have
$T$-tractability. That is, for some $\vp \ge \vzero$, there exists a positive constant $C_{\vp}$ such that
\[
\comp(\varepsilon,d)\le C_{\vp}\, T(\varepsilon^{-1},d,\vp) \qquad \forall \varepsilon > 0,  \ d \in \naturals.
\]
Since the sequence of singular values $\lambda_{1,d}, \lambda_{2,d}, \ldots $ is non-increasing, we have
\begin{equation}\label{eq:lambda_K4}
	\lambda_{\lfloor C_{\vp}\, T(\varepsilon^{-1},d,\vp)\rfloor +1,d}\le \varepsilon \qquad \forall \varepsilon > 0,  \ d \in \naturals.
\end{equation}
Define the integers
\[
i (\varepsilon,d,\vp):= \lfloor C_{\vp}\, T(\varepsilon^{-1},d,\vp)\rfloor +1, \qquad
i (\infty,d,\vp) := \lfloor C_{\vp}\, T(0,d,\vp)\rfloor +1 \ge 1.
\]
Thus, it follows by \eqref{eq:lambda_K4} that $\lambda_{i(\varepsilon,d,\vp),d} \le \varepsilon$.
Note furthermore that we always have
\[
i(\varepsilon,d,\vp)\le C_{\vp}\, T(\varepsilon^{-1},d,\vp)+1 \le C_{\vp} T(\lambda_{i(\varepsilon,d,\vp),d}^{-1},d,\vp)+1 \qquad \forall \varepsilon > 0, \ d \in \naturals,
\]
since
$T(\cdot,d,\vp)$ is non-decreasing.

For $\varepsilon$ taking on all positive values, $i(\varepsilon,d,\vp)$ takes on, by \eqref{eq:Tlim}, all values greater than or equal to $i(\infty,d,\vp)$, so
\[
i\le C_{\vp}\, T(\varepsilon^{-1},d,\vp)+1 \le C_{\vp} T(\lambda_{i,d}^{-1},d,\vp)+1 \qquad \forall i \ge i(\infty,d,\vp),  \ d \in \naturals.
\]
This implies via our technical assumption \eqref{eq:ptauassume} that for any $\tau > 1$,
\begin{gather*}
	K_{\vp,\tau}\,T (\lambda_{i,d}^{-1},d,\tau \vp) \ge
	[T(\lambda_{i,d}^{-1},d, \vp)]^\tau
	\ge
	\left[\frac{(i-1)}{C_{\vp}}\right]^\tau \qquad \forall i \ge i(\infty,d,\vp),  \ d \in \naturals, \\
	 \frac{1}{T (\lambda_{i,d}^{-1},d,\tau \vp)} \le
	\frac{K_{\vp,\tau}\, C_{\vp}^\tau}{(i-1)^\tau}\qquad \forall i \ge \max\{i(\infty,d,\vp),2\},  \ d \in \naturals.
\end{gather*}

Let $\vp' = \tau \vp$.  Note that $i(\infty,d,\cdot)$ is non-decreasing.  Summing both sides of the latter inequality over $i$ from $i(\infty,d,\vp') + 1 \ge i(\infty,d, \vp) + 1 \ge 2$ to $\infty$ it follows that
\begin{equation*}
	\sup_{d\in\naturals} \, \, \sum_{i=i (\infty,d,\vp')+1}^\infty \frac{1}{T (\lambda_{i,d}^{-1},d,\vp')}
	 \le  K_{\vp,\tau}\, C_{\vp}^\tau
	\sum_{i=2}^\infty \frac{1}{(i-1)^\tau} \\
	 \le  K_{\vp,\tau}\, C_{\vp}^\tau
	\zeta (\tau)  < \infty,
\end{equation*}
where $\zeta$ denotes the Riemann zeta function.

Note now that, for any $d\in\naturals$,
\begin{align*}
 i(\infty,d,\vp') +1  & = \lfloor C_{\vp'}\, T(0,d,\vp')\rfloor +2 \le  \left[C_{\vp'} + \frac{2}{T(0,d,\vp')}   \right]\, T(0,d,\vp')\\
 & \le  \underbrace{\left[C_{\vp'} + \frac{2}{T(0,1,\vp')}   \right]}_{=:L_{\vp'}}\, T(0,d,\vp')
 \le \lceil L_{\vp'} T(0,d,\vp')\rceil .
\end{align*}
For this choice of $L_{\vp'}$ we have
\[
\sup_{d\in\naturals}\,\, \sum_{i=\lceil L_{\vp'}T(0,d,\vp')\rceil}^\infty \frac{1}{T(\lambda_{i,d}^{-1},d, \vp')}
\le
\sup_{d\in\naturals} \, \, \sum_{i=i (\infty,d,\vp')+1}^\infty \frac{1}{T (\lambda_{i,d}^{-1},d,\vp')} <
\infty.
\]
This yields \eqref{eq:tractiff4} with $\vp$ replaced by $\vp'$, so we see the necessity of $\eqref{eq:tractiff4}$.

\bigskip

\noindent \textbf{Optimality:} \\
This proof is similar to the proof for the optimality condition for strong tractability in Theorem \ref{thm_main_strong_tract2}. We must show that $\cp_{\textup{trct}} = \widetilde{\cp}_{\textup{trct}}$.  Then, the expression for  $\cp_{\textup{opt}}$ in this theorem automatically follows from its definition in \eqref{eq:Poptdef}.

First we show  that if $\vp^* \in \cp_{\textup{trct}}$, then $\vp^* \in \widetilde{\cp}_{\textup{trct}}$.  If $\vp^* \in \cp_{\textup{trct}}$, then \eqref{eq:tract_def} must hold for all $\vp' \in (\vp^*,\vinfty)$.  For any $\vp' \in (\vp^*,\vinfty)$, we may choose $\vp$ and $\tau$ such that $\vp^* < \vp < \vp' = \tau \vp$. Since $\vp > \vp^*$, \eqref{eq:tract_def} also holds for $\vp$, and it follows that \eqref{eq:tractiff4} must hold for $\vp'$ by the proof of necessity above.  Thus, $\vp^* \in \widetilde{\cp}_{\textup{trct}}$.

Next we show that if $\vp^* \in \widetilde{\cp}_{\textup{trct}}$, then $\vp^* \in \cp_{\textup{trct}}$.  If $\vp^* \in \widetilde{\cp}_{\textup{trct}}$, then \eqref{eq:tractiff4}  holds for all $\vp \in (\vp^*,\vinfty)$.  By the argument to prove the sufficient condition above, it follows that \eqref{eq:tract_def} must also hold for all $\vp \in (\vp^*,\vinfty)$. Thus, $\vp^* \in \cp_{\textup{trct}}$.

\bigskip
\noindent This concludes the proof of Theorem \ref{thm_main_tract2}.  
\end{proof}

\section{Examples} \label{sec:examples}
In this section, we present six examples relating to Theorem \ref{thm_main_strong_tract2} and Theorem \ref{thm_main_tract2}, both of which were shown in Section \ref{sec:spt}, for the various notions of tractability that are listed in Table \ref{tab:commonT}. In the subsequent sections, we will continue to present various examples.


The concept of (strong) polynomial tractability for both the algebraic case and the exponential case will be first considered as follows:


%

\begin{example}[Algebraic polynomial tractability]
Let the tractability function, $T$,
be defined by
\[
 T(\varepsilon^{-1},d,\vp)= \max\{1,\varepsilon^{-p}\}\, d^q
 \qquad \forall \varepsilon > 0, \,  d \in \naturals, \, \vp \in [0,\infty)^2,
\] where in this case $\vp = (p,q)$.
Then Theorem \ref{thm_main_tract2} yields the  equivalent condition for algebraic polynomial tractability:
\[
 \sup_{d\in\naturals}\, d^{-q} \sum_{i=\lceil L_{(p,q)} d^{q}\rceil}^\infty \min\{1,\lambda_{i,d}^p\}  < \infty \text{ for some } p,q \ge 0, L_{(p,q)} >0.
\]
Furthermore, for optimality, let $\widetilde{\cp}_{\textup{ALG-trct}} : = \{(p^*,q^*) : \text{the above condition holds }\forall (p,q) \in (p^*,\infty)\times(q^*,\infty) \}$. Therefore, the set of optimal $(p,q)$ is the set 
\[
    \{(p^*,q^*) \in \widetilde{\cp}_{\textup{ALG-trct}} : (p^*,q^*) \notin [\widetilde{p},\infty) \times [\widetilde{q},\infty)\, \, \forall (\widetilde{p},\widetilde{q}) \in \widetilde{\cp}_{\textup{ALG-trct}}\setminus \{p^*,q^*\}  \}.
\]
This essentially recovers the result on polynomial tractability in \cite[Theorem 5.1]{NovWoz08a}. 

Having defined the tractability function as above, we will now examine a specific case where we always set $d=1$, i.e.,
\[
T(\varepsilon^{-1},1,\vp)= \max\{1,\varepsilon^{-\vp}\}\qquad \forall \varepsilon > 0, \,  d \in \naturals, \, \vp \in [0,\infty),
\]
where in this case $\vp=p$ is a scalar.

Then Theorem \ref{thm_main_strong_tract2} yields the equivalent condition for algebraic strong polynomial tractability:
\[
\sup_{d\in\naturals} \sum_{i=L_p}^\infty \min\{1,\lambda_{i,d}^p\} < \infty \text{ for some } p \ge 0, L_p \in \naturals.
\]
Furthermore, the optimal $p$ is the infimum of all $p$ for which the above condition holds. This essentially recovers the result in \cite[Theorem 5.1]{NovWoz08a}.
\end{example}
Next, we study the exponential case where we replace $\varepsilon^{-1}$ by $\log(1+\varepsilon^{-1})$ and consider the same notions of tractability as for the algebraic case.
\begin{example}[Exponential polynomial tractability]
Let the tractability function, $T$,
be defined by
\[
 T(\varepsilon^{-1},d,\vp)=\left[\max\{1,\log(1+\varepsilon^{-1})\}\right]^p\, d^q \qquad \forall \varepsilon > 0, \,  d \in \naturals, \, \vp \in [0,\infty)^2,
\] where $\vp = (p,q)$. Then Theorem \ref{thm_main_tract2} yields  the equivalent condition for exponential polynomial tractability:
\begin{equation}\label{eq:exp_pt}
 \sup_{d\in\naturals}\, d^{-q} \sum_{i=\lceil L_{(p,q)} d^{q}\rceil}^\infty \frac{1}{[\max\{1,\log(1+\lambda_{i,d}^{-1})\}]^p} < \infty  \text{ for some } p,q \ge 0, L_{(p,q)} >0.
\end{equation}
Furthermore, for optimality, let $\widetilde{\cp}_{\textup{EXP-trct}} : = \{(p^*,q^*) : \text{the above condition holds }\forall (p,q) \in (p^*,\infty)\times(q^*,\infty) \}$. Therefore, the set of optimal $(p,q)$ is the set 
\[
    \{(p^*,q^*) \in \widetilde{\cp}_{\textup{EXP-trct}} : (p^*,q^*) \notin [\widetilde{p},\infty) \times [\widetilde{q},\infty)\, \, \forall (\widetilde{p},\widetilde{q}) \in \widetilde{\cp}_{\textup{EXP-trct}}\setminus \{p^*,q^*\}  \}.
\]
Again, we will look at the special case in which we always set $d=1$. In this case, we will define $T$ as follows.
\[
T(\varepsilon^{-1},1,p)= [\max\{1,\log(1+\varepsilon^{-1})\}]^{\vp} \quad \forall \varepsilon > 0, \,  d \in \naturals, \, p \in [0,\infty),\]
where $\vp=p$ is a scalar. Then Theorem \ref{thm_main_strong_tract2} yields the equivalent condition for exponential strong polynomial tractability:
\begin{equation}\label{eq:exp_spt}
\sup_{d\in\naturals} \sum_{i=L_{p} }^\infty \frac{1}{[\max\{1,\log(1+\lambda_{i,d}^{-1})\}]^p}< \infty,\quad  \text{ for some } p \ge 0, L_{p} \in \naturals.
\end{equation}
In this case, the optimal $p$ is the infimum of all $p$ for which the above condition holds.
\end{example}

In the paper \cite{KriWoz19a}, the authors derived conditions equivalent to exponential strong polynomial tractability and exponential polynomial tractability, respectively. These conditions read (with notation slightly adapted to our present paper) as follows:
\begin{itemize}
    \item We have exponential strong polynomial tractability iff there exists a $\tau>0$ and
    a $C\in\naturals$ such that
    \begin{equation}\label{eq:KW_exp_spt}
        \sup_{d\in\naturals} \sum_{i=C}^{\infty} \lambda_{i,d}^{i^{-\tau}} < \infty.
    \end{equation}
    \item We have exponential polynomial tractability iff there exist $\tau_1,\tau_3\ge 0$ and $\tau_2, C>0$ such that
    \begin{equation}\label{eq:KW_exp_pt}
        \sup_{d\in\naturals} d^{-\tau_1} \sum_{i=\lceil C d^{\tau_3} \rceil}^{\infty} \lambda_{i,d}^{i^{-\tau_2}} < \infty.
    \end{equation}
\end{itemize}

Firstly, we remark that the role of $L_p$ in  \eqref{eq:exp_spt} corresponds to that of $C$ in \eqref{eq:KW_exp_spt}.
Moreover, $q$ in \eqref{eq:exp_pt} has a similar role as $\tau_1$ and $\tau_3$ in \eqref{eq:KW_exp_pt} (i.e., \eqref{eq:KW_exp_pt} is slightly more precise than \eqref{eq:exp_pt} regarding the exponents of $d$, but this is just a minor difference); furthermore $L_{(p,q)}$ in \eqref{eq:exp_pt} corresponds to $C$ in \eqref{eq:KW_exp_pt}.

Secondly, it should be noted that our condition \eqref{eq:exp_spt} is easier to check in practice than condition \eqref{eq:KW_exp_spt}. Since in \eqref{eq:KW_exp_spt} the summation indices $i$ show up in the exponents of the summands, we need the exact ordering of the singular values $\lambda_{i,d}$, whereas in \eqref{eq:exp_spt} we only have a summation of the values of $1/T(\lambda_{i,d}^{-1},d,\vp)$, and the exact order of the singular values can be neglected. A similar
observation holds for \eqref{eq:exp_pt} and \eqref{eq:KW_exp_pt}.

The following proposition states that the conditions \eqref{eq:exp_spt} and \eqref{eq:KW_exp_spt}, as well as the conditions \eqref{eq:exp_pt} and \eqref{eq:KW_exp_pt}, are indeed equivalent. 
The proof of this result is deferred to Appendix 2 where we will prove this statement for the latter two conditions, i.e., for the conditions on exponential polynomial tractability. The case of exponential strong polynomial tractability is not explicitly included, as it is actually simpler and can be treated analogously.  We give a short example of how we can apply the condition \eqref{eq:exp_spt} to a concrete problem in Appendix 1.

\begin{proposition}\label{prop:equiv_conditions}
The conditions \eqref{eq:exp_spt} and \eqref{eq:KW_exp_spt} are equivalent. Furthermore, the conditions \eqref{eq:exp_pt} and \eqref{eq:KW_exp_pt} are equivalent.
\end{proposition}

Next, we provide a generalization of the above examples of tractability. We would like to emphasize that this example was previously known (see, e.g., \cite{GW06}) where in this case the roles of $\varepsilon^{-1}$ and $d$ are separated. In order to accomplish this, we will first define two functions $\phi$ and $\psi$ with positive values that are increasing in both of their arguments. The equivalent conditions for this example take the same form as those in Theorems \ref{thm_main_strong_tract2} and \ref{thm_main_tract2}.

\begin{example}[(Strong) Separable tractability]
Suppose that the tractability function is defined by
\[
 T(\varepsilon^{-1},d,\vp)= \phi(\varepsilon^{-1},p)\, \psi(d,q)
 \qquad \forall \varepsilon > 0, \,  d \in \naturals, \, \vp \in [0,\infty)^2,
\]
with non-decreasing functions $\phi : (0,\infty) \times (0,\infty) \to (0,\infty)$ and $\psi : \naturals  \times (0,\infty) \to (0,\infty)$ and where $\vp = (p,q)$, hence $T$ is non-decreasing in all variables.

We say that a problem is $T$-separably tractable with parameter $\vp$ iff there exists a positive constant $C_{\vp}$ which does not depend on $\varepsilon$ and $d$ such that
\[
\comp(\varepsilon,d) \leq  C_{\vp} \,  \phi(\varepsilon^{-1},p)\, \psi(d,q)\qquad \forall \varepsilon >0, \, d\in\naturals.
\]
A problem is strongly $T$-separably tractable with parameter $\vp$ iff the information complexity is independent of $d$, that is, there exists a positive constant $C_{\vp}$ which again does not depend on $\varepsilon$ and $d$ such that
\[
\comp(\varepsilon,d) \leq C_{\vp}\,\phi(\varepsilon^{-1},p)\qquad \forall \varepsilon >0, \, d\in\naturals,
\] where in this case $\vp = p$ is a scalar.
Furthermore, we assume the following, which corresponds to the conditions in \eqref{eq:Tlim}, \eqref{eq:T0}, and \eqref{eq:ptauassume}. We assume that
\[
\lim_{\varepsilon\rightarrow0}\phi(\varepsilon^{-1},p) = \infty \qquad \forall d\in\naturals, p >0,
\]
since there are an infinite number of positive singular values. Furthermore, we have that $\phi(\cdot,p)$ is non-decreasing, hence we define the following limit
\[
\lim_{\varepsilon \rightarrow \infty}\phi(\varepsilon^{-1},p) >0.
\]
Additionally, we have the following technical assumptions. Suppose that, for any real $\tau_1, \tau_2>0$ and any $p,q>0$, we have that
\begin{align*}
     [\phi(\varepsilon^{-1},p)]^{\tau_1} \le \hat{K}_{p,\tau_1} \phi (\varepsilon^{-1},\tau_1 p)
     \quad \mbox{and} \quad
      [\psi(d,q)]^{\tau_2} \le \hat{K}_{q,\tau_2} \psi (d,\tau_2 q) \qquad \forall \, \varepsilon >0,\, d\in \naturals,
\end{align*} where $\hat{K}_{p,\tau_1}>0$ is a constant that may depend on $p$ and $\tau_1$, but is independent of $\varepsilon$, and where $\hat{K}_{q,\tau_2}>0$ is a constant that may depend on $q$ and $\tau_2$, but is independent of $d$.

Theorem \ref{thm_main_strong_tract2} yields the following equivalent condition for strong $T$-separable tractability:
\begin{equation*} \label{eq:strong_tractiff}
     \sup_{d \in \naturals} \sum_{i = L_p}^\infty \frac{1}{\phi(\lambda_{i, d}^{-1},p)} < \infty,
\end{equation*}
for some $p>0$ and $L_p \in \naturals$. Furthermore, the optimal $p$ is also the infimum of all $p$ satisfying the latter condition. Also,
Theorem \ref{thm_main_tract2} yields the following equivalent condition for $T$-separable tractability:
\[
 \sup_{d \in \naturals} \frac{1}{\psi(d,q)} \sum_{i = \lceil L_{(p,q)}\psi(d,q)\rceil}^\infty \frac{1}{\phi(\lambda_{i, d}^{-1},p)} < \infty,
\] for some $p,q >0$ and a positive constant $L_{(p,q)}$.
\end{example}
Next, we introduce another notion of tractability called \emph{non-separable tractability}. See \cite{GW08}, which is probably where the term ``non-separable tractability" was first introduced.
In the following, we will define two functions $\hat{\phi}$ and $\hat{\psi}$ that take the roles of $\varepsilon^{-1}$ and $d$. As will be seen below, this kind of tractability function includes the special case of the so-called \emph{quasi-polynomial tractability}.
\
\begin{example}[Non-separable tractability]
    Suppose we define the tractability function $T$ by
\[
 T(\varepsilon^{-1},d,p)= \exp (p \, \hat{\phi}(\varepsilon^{-1})\, \hat{\psi}(d)) = [\exp( \hat{\phi}(\varepsilon^{-1}))]^{p \, \hat{\psi}(d)}
 \qquad \forall \varepsilon > 0, \,  d \in \naturals, \, p\in [0,\infty),
\]  with non-decreasing functions $\hat{\phi} : (0,\infty) \to (0,\infty)$ and $\hat{\psi} : \naturals \to (0,\infty)$.
We say that a problem is non-separably tractable with scalar parameter $p$ iff there exists a positive $C_{p}$ such that

\[\comp(\varepsilon,d)
\leq  C_{p}\exp (p(\hat{\phi}(\varepsilon^{-1})\hat{\psi}(d))).\]

Hence we can apply Theorem \ref{thm_main_tract2} as follows. We have the following equivalent condition for non-separable tractability:
\[
\sup_{d\in\mathbb{N}}\sum_{i=\lceil L_p\,\hat{\phi}(0)^{p \, \hat{\psi}(d)}\rceil} [\exp ( \hat{\phi}(\lambda_{i,d}^{-1}))]^{-p \, \hat{\psi}(d)}\quad \text{for some } p \geq 0, L_p >0.
\]

\end{example}
For $\hat{\phi}$ and $\hat{\psi}$ of special form, the condition above can be simplified further as in the following two special cases of non-separable tractability.

\begin{example}[Algebraic quasi-polynomial tractability]
Let the tractability function, $T$,
be defined by
\[
 T(\varepsilon^{-1},d,p)= \exp\{ p(1+\log(\max\{1,\varepsilon^{-1}\}))(1+\log (d))\}  \qquad \forall \varepsilon > 0, \,  d \in \naturals, \, p \in [0,\infty).
\] We note that the choice of $T$ is equivalent to the definition in \cite{GneWoz11a}. This can be rewritten as
\[
T(\varepsilon^{-1},d,p) = e^p(\max\{1,\varepsilon^{-1}\})^{p(1+\log (d))}d^p.
\]

Then Theorem \ref{thm_main_tract2} yields the following equivalent condition for quasi-polynomial tractability in the algebraic case:
\[
\sup_{d\in\mathbb{N}}\,d^{-p}\sum_{i=\lceil L_p\,d^{p}\rceil} \frac{1}{\left(\max\{1,\lambda_{i,d}\}\right)^{p(1+\log (d))}} < \infty \qquad \text{for some } p \geq 0, L_p >0.\]
\end{example}
This essentially recovers \cite[Theorem 23.1]{NovWoz12a}.

Next, we will consider the exponential case as follows.
\begin{example}[Exponential quasi-polynomial tractability]

\noindent Suppose that the tractability function, $T$, is defined by
\[
 T(\varepsilon^{-1},d,p)= \exp \{p(1+\log(\max\{1,\log(1+\varepsilon^{-1})\}))(1+\log (d))\}  \qquad \forall \varepsilon > 0, \,  d \in \naturals, \, p\in [0,\infty).
\] This can be rewritten as
\[
T(\varepsilon^{-1},d,p) = e^p(\max\{1,\log(1+\varepsilon^{-1})\})^{p(1+\log (d))}d^p.
\]
Hence we can apply Theorem \ref{thm_main_tract2} as follows. We have the following equivalent condition for quasi-polynomial tractability in the exponential case:
\[
\sup_{d\in\mathbb{N}}\,d^{-p}\sum_{i=\lceil L_p\,d^{p}\rceil} \frac{1}{\left(\max\{1,\log(1+\lambda_{i,d}^{-1})\} \right)^{p(1+\log (d))}}< \infty \quad \text{for some } p \geq 0, L_p >0.
\]
This essentially recovers \cite[Theorem 2]{KriWoz19a}.
\end{example}


\section{Sub-$h$ Tractability}\label{sec:subh}
In this section we generalize the concept of weak tractability (see again \cite{NovWoz08a}--\cite{NovWoz12a}).  Let  $h:[0,\infty) \to [1,\infty)$ satisfy
\begin{equation} \label{eq:h_cond}
	h \text{ is strictly increasing}, \qquad
	h(0)=1, \qquad
 h(x+y) \ge  h(x)\cdot h(y)  \quad \forall x,y \ge 0.
\end{equation}
This means that $\log(h)$ is a superadditive function.  An example of the function $h$ is the exponential function, from which we can recover weak tractability.

Furthermore, let us denote the inverse function of $h$ by $h^{-1}$.
Then, for arbitrary $z,w \in [1,\infty)$ we choose $x=h^{-1}(z)$ and $y=h^{-1}(w)$, to obtain from \eqref{eq:h_cond}
\[
z w \le h (h^{-1}(z)+h^{-1}(w)),
\]
which yields by the monotonicity of $h^{-1}$ that
\begin{equation}\label{eq:submult_h_inverse}
h^{-1}(z w) \le h^{-1}(z)+h^{-1}(w).
\end{equation}
Now we define a generalized notion of weak tractability.
\begin{definition} \label{def:subhT}
	A problem is sub-$h$-$T$ tractable for parameter $\vp > \vzero$ if
	\begin{equation*}
		\lim_{\varepsilon^{-1} + d \to \infty} \frac{h^{-1}(\max(1,\comp(\varepsilon,d)))}{T(\varepsilon^{-1},d,\vp)} = 0.
	\end{equation*}
\end{definition}

Note that this definition of sub-$h$-$T$ tractability implies that
\begin{equation} \label{eq:comp_slow_hcT}
		\lim_{\varepsilon^{-1} + d \to \infty} \frac{\comp(\varepsilon,d)}{h(cT(\varepsilon^{-1},d,\vp))} = 0 \qquad \forall c > 0,
\end{equation}
which means that for $\varepsilon \to 0$ or $d \to \infty$  $\comp(\varepsilon,d)$ must increase slower than $h(cT(\varepsilon^{-1},d,\vp))$, no matter how small $c$ is.

\begin{theorem}\label{thm:subhT}
	Let $T$ be a tractability function as specified in \eqref{eq:Tspec} and satisfying \eqref{eq:Tconditions}.  Moreover, assume that 
    $\lim_{\varepsilon^{-1} + d \to \infty} T(\varepsilon^{-1},d,\vp) = \infty$
    for all $\vp > \vzero$.  The problem is sub-$h$-$T$ tractable if and only if
	\[
	S_c = S_{c,\vp} :=\sup_{d\in\naturals}  \, \sum_{i=1}^\infty \frac1{h(c\, T(\lambda_{i,d}^{-1},d,\vp))} <\infty \qquad \forall c > 0.
	\]
\end{theorem}
\begin{proof}
The idea of this proof follows \cite{WerWoz17a}.
We suppress the $\vp$ dependence in the proof, because it is insignificant.

\bigskip

\noindent \textbf{Sufficient condition:}\\
Suppose that $S_c <\infty$ for all $c>0$.  Since the terms of the series in the definition of $S_c$ are non-increasing, we have
	\[
		\frac{n}{h(cT(\lambda_{n,d}^{-1},d))}\le S_c \qquad \forall n,d \in \naturals,
	\]
or equivalently
\begin{equation} \label{eq:Sclb}
	n\le S_c \, h(cT(\lambda_{n,d}^{-1},d))   \qquad \forall n,d \in \naturals.
\end{equation}
This implies that for the special choice
\[
	n = \left\lceil
	S_c \, h(cT(\varepsilon^{-1},d))
	\right\rceil \ge 1
\]
	we have
	\[
	T(\lambda_{n,d}^{-1},d) \ge T(\varepsilon^{-1},d) \qquad \forall \varepsilon > 0, \, d \in \naturals.
	\]
Since $T(\cdot,d)$ is non-decreasing, it follows that $\lambda_{n,d} \le \varepsilon$, and thus
\[
\comp(\varepsilon,d) \le  \left\lceil
S_c \, h(cT(\varepsilon^{-1},d))
\right\rceil \le   \lceil
S_c \, h(cT(\lambda_{1,1}^{-1},1)) \, h(cT(\varepsilon^{-1},d))
\rceil\qquad \forall c, \varepsilon > 0, \, d \in \naturals
\]
because $h$ is no smaller than $1$.

As $S_c \ge 1/h(cT(\lambda_{1,1}^{-1},1))$ by \eqref{eq:Sclb}, it follows that $S_c \, h(cT(0,1)) \, h(cT(\varepsilon^{-1},d))  \ge 1$.  Define $\widetilde{S_c}: = 2 S_c \, h(cT(\lambda_{1,1}^{-1},1))$, and note that $\widetilde{S}_c$ must also be finite for all $c > 0$.  It follows that
\begin{equation*}
\max(1,\comp(\varepsilon,d))
 \le \lceil
(\widetilde{S}_c/2) \, h(cT(\varepsilon^{-1},d))
\rceil  \le   \widetilde{S}_c \, h(cT(\varepsilon^{-1},d))
\qquad\forall c, \varepsilon > 0, \, d \in \naturals.
\end{equation*}
Since $h$ is strictly increasing, so is $h^{-1}$, and thus by \eqref{eq:submult_h_inverse},
\begin{align*}
h^{-1}(\max(1,\comp(\varepsilon,d)) ) & \le  h^{-1} \left(
\widetilde{S}_c \, h(cT(\varepsilon^{-1},d))  \right)
\\
&  \le  h^{-1}  (\widetilde{S}_c) + cT (\varepsilon^{-1},d)
\quad \forall c, \varepsilon > 0, \, d \in \naturals.
\end{align*}

By the hypothesis of this theorem, $\lim_{\varepsilon^{-1} + d \to \infty} T(\varepsilon^{-1},d) = \infty$. Therefore
\begin{equation*}
	\lim_{\varepsilon^{-1} + d \to \infty} \frac{h^{-1}(\max(1,\comp(\varepsilon,d)))}{T(\varepsilon^{-1},d)}
	\le \lim_{\varepsilon^{-1} + d \to \infty} \frac{h^{-1} (2 S_c)}{T(\varepsilon^{-1},d)}  + c = c \qquad\forall c >  0.
\end{equation*}
Since this limit is bounded above by all positive $c$, it must be zero, and the problem is sub-$h$-$T$ tractable.

\bigskip

\noindent \textbf{Necessary condition:}\\
Suppose that the problem is  sub-$h$-$T$ tractable. Then by \eqref{eq:comp_slow_hcT}, for any $c>0$ there exists a positive integer $V_c$ such that
\begin{equation*}
	\comp(\varepsilon,d) \le \left\lfloor h(cT(\varepsilon^{-1},d)) \right\rfloor \qquad \forall \varepsilon^{-1} + d \ge V_c.
\end{equation*}
By the definition of the information complexity in \eqref{eq:comp_def}, it follows that
\[
	\lambda_{n(\varepsilon,d,c),d} \le \varepsilon \quad \forall \varepsilon^{-1} + d \ge V_c, \qquad \text{where }
	n(\varepsilon,d,c):=\left\lfloor h(cT(\varepsilon^{-1},d)) \right\rfloor +1.
\]
Since $h(cT(\cdot,d))$ is non-decreasing, it follows that
\begin{equation}
		n(\varepsilon,d,c)\le  h(cT(\varepsilon^{-1},d)) +1\le  h(cT(\lambda^{-1}_{n(\varepsilon,d,c),d},d)) +1\qquad  \quad \forall \varepsilon^{-1} + d \ge V_c. \label{eq:nedcbound}
\end{equation}

For a fixed $c$ and $d$, define
\begin{equation} \label{eq:nstardef}
	\varepsilon^{-1}_{\max}(d,c) := \max(V_c - d,\lambda_{1,d}^{-1}), \qquad n^*(d,c) : = n(\varepsilon^{-1}_{\max}(d,c),d,c) \le h(cT(\varepsilon^{-1}_{\max}(d,c),d)) +1.
\end{equation}
By varying $\varepsilon^{-1}$ in the interval $[\varepsilon^{-1}_{\max}(d,c),\infty)$,  the integer $n(\varepsilon,d,c)$ attains, by Assumption \eqref{eq:Tlim}, all values greater than or equal to $n^*(d,c)$.  Thus, by \eqref{eq:nedcbound},
\begin{align}
		n &\le h(cT(\lambda^{-1}_{n,d},d)) \qquad  \forall n \ge n^*(d,c), \label{eq:subhTneq}\\
		\nonumber
		\frac{1}{h(2 cT(\lambda^{-1}_{n,d},d)) } & \le \frac{1}{[h(cT(\lambda^{-1}_{n,d},d))]^{2}}  \le \frac 1{n^2}   \qquad  \forall n \ge n^*(d,c) \qquad \text{by \eqref{eq:h_cond}}. \\
		\intertext{Combining this inequality with the upper bound on $n^*(d,c)$ yields}
		\nonumber
		S_{2c} & =\sup_{d\in\naturals}  \, \sum_{i=1}^\infty \frac1{h(2cT(\lambda_{i,d}^{-1},d))} \\
		\nonumber
		&  = \sup_{d\in\naturals} \left [
		\sum_{i=1}^{n^*(d,c)-1} \frac1{h(2cT(\lambda_{i,d}^{-1},d))}
		+ \sum_{i=n^*(d,c)}^\infty \frac1{h(2cT(\lambda_{i,d}^{-1},d))}
		\right]\\
		 \nonumber
		 & \le \sup_{d\in\naturals} \left [
\frac{n^*(d,c)-1}{h(2cT(\lambda_{1,d}^{-1},d))}
+ \sum_{i=1}^\infty \frac1{i^2}
\right]
 \qquad \text{by \eqref{eq:subhTneq}}\\
 \label{eq:lastone}
		 & \le \sup_{d\in\naturals} \left [
		 \frac{ h(cT(\varepsilon^{-1}_{\max}(d,c),d))}{h(2cT(\lambda_{1,d}^{-1},d))}
		\right] + \frac{\pi^2}{6}
		\qquad \text{by \eqref{eq:nstardef}.}
\end{align}

Note that
\[
\varepsilon^{-1}_{\max}(d,c) = \max(V_c - d,\lambda_{1,d}^{-1})
= \begin{cases} V_c - d \le V_c, & d \in \{1, \ldots, d^*_c : = \lfloor V_c - \lambda_{1,d}^{-1} \rfloor\}, \\
\lambda_{1,d}^{-1}, & d \in \{d^*_c+1, \ldots\},
\end{cases}
\]
and furthermore, it follows from above that we always have $d^*_c\le \lfloor V_c \rfloor$.
Recalling that $T$ is non-decreasing in its arguments, this implies that
\begin{align*}
    \MoveEqLeft{\sup_{d\in\naturals} \,
		 \frac{ h(cT(\varepsilon^{-1}_{\max}(d,c),d))}{h(2cT(\lambda_{1,d}^{-1},d))}}\\
  &\le \max \left \{
  \sup_{d \in \{1, \ldots, d^*_c\}} \,
		 \frac{ h(cT(\varepsilon^{-1}_{\max}(d,c),d))}{h(2cT(\lambda_{1,d}^{-1},d))},
  \sup_{d\in \{d^*_c+1, \ldots \}} \,
		 \frac{ h(cT(\varepsilon^{-1}_{\max}(d,c),d))}{h(2cT(\lambda_{1,d}^{-1},d))}
  \right \} \\
  & \le \max \left \{
  \sup_{d \in \{1, \ldots, d^*_c\}}  \,
		 \frac{ h(cT(V_c,d))}{h(2cT(\lambda_{1,d}^{-1},d))}
  ,
  \sup_{d\in \{d^*_c+1, \ldots \}}  \,
		 \frac{ h(cT(\lambda_{1,d}^{-1},d))}{h(2cT(\lambda_{1,d}^{-1},d))}
  \right \}
  \\
  & \le \max \left \{
		 \frac{ h(cT(V_c,d^*_c))}{1},
  \sup_{d\in \{d^*_c+1, \ldots \}}  \,
		 \frac{ h(cT(\lambda_{1,d}^{-1},d))}{[h(cT(\lambda_{1,d}^{-1},d))]^2}
  \right \} \\
  & \le \max \left \{
		 \frac{ h(cT(V_c,\lfloor d^*_c \rfloor))}{1},
		 1
  \right \} \\
  & = h(cT(V_c,\lfloor V_c \rfloor))
		  \\
  & < \infty,
\end{align*}
where we used that $h(\cdot)$ is always at least 1. Combining this bound with the upper bound on $S_c$ in \eqref{eq:lastone} establishes the finiteness of $S_{2c}$ and completes the proof of the necessary condition.
This concludes the proof of Theorem \ref{thm:subhT}.

\end{proof}

\begin{example}[Weak Tractability]
	Suppose that
	\[
	T(\varepsilon^{-1},d,s,t) = \max(1,\varepsilon^{-1})^s + d^t \quad \forall \varepsilon,s,t > 0, \, d \in \naturals, \qquad h(x) = \exp(x) \quad \forall x \ge 0.
	\]
	Then Theorem \ref{thm:subhT} implies that
	\[\lim_{\varepsilon^{-1}+d\rightarrow \infty} \frac{\log( \comp(\varepsilon,d))}{\varepsilon^{-s}+d^t} = 0
	\iff
	 \sup_{d\in\naturals}  \, \sum_{i=1}^\infty \exp(-c\, (\min\{1,\lambda_{i,d}\}^{s} + d^t )) <\infty \qquad \forall c > 0.
	\]
	The expression on the left is the notion of $(s,t)$-weak tractability introduced in \cite{SW15}, and the condition on the right was derived in \cite[Theorem 3.1]{WerWoz17a}.  Weak tractability corresponds to $(s,t)=(1,1)$, and uniform weak tractability corresponds to the case when a problem is weakly tractable for all positive $s$ and $t$.
\end{example}

\section{(Strong) Tractability on a Restricted Domain} \label{sec:restricted}

Suppose that the domain of interest $(\varepsilon^{-1},d)$ is not all of $(0,\infty) \times \naturals$ but some subset
\begin{equation}
    \Omega  := \{(\varepsilon^{-1},d) : \varepsilon \in (\theM(d),\infty), \ d \in \naturals \}, \qquad \text{where }\theM : \naturals \to [0,\infty).
\end{equation}
We then expect that the equivalent conditions for (strong) tractability will be similar in form, but weaker than the conditions in Theorems  \ref{thm_main_strong_tract2} and \ref{thm_main_tract2}.

\begin{definition}
    A problem is strongly
$T$-tractable with parameter $\vp$ \emph{on the restricted domain $\Omega$} iff the information complexity is independent of the dimension of the problem, that is, there exists a positive constant $C_{\vp}$, again independent of $\varepsilon$ and $d$, such that
\begin{equation} \label{eq:rest_strong_tract_def}
	\comp(\varepsilon,d) \le C_{\vp}\, T(\varepsilon^{-1},1,\vp) \qquad \forall(\varepsilon^{-1},d) \in \Omega.
\end{equation}

A problem is
$T$-tractable with parameter $\vp$ \emph{on the restricted domain $\Omega$} iff there exists a positive constant $C_{\vp}$, again independent of $\varepsilon$ and $d$, such that
\begin{equation} \label{eq:rest_tract_def}
	\comp(\varepsilon,d) \le C_{\vp}\, T(\varepsilon^{-1},d,\vp) \qquad \forall(\varepsilon^{-1},d) \in \Omega.
\end{equation}
 Define the closures of the sets of parameters for which our (strong) tractability conditions hold:
\begin{equation*}
	\cp_{\textup{rtrct}} : = \{\vp^* : \eqref{eq:rest_tract_def} \text{ holds }\forall \vp \in (\vp^*,\boldsymbol{\infty})\}, \qquad
	\cp_{\textup{rstrct}} : = \{\vp^* : \eqref{eq:rest_strong_tract_def} \text{ holds }\forall \vp \in (\vp^*,\boldsymbol{\infty})\}.
\end{equation*}
    The  set of optimal parameters is defined as all of those parameters satisfying the (strong) tractability conditions that are not greater than or equal to others:
\begin{gather}
	\mathcal{P}_{\textup{ropt}} : = \{\vp^* \in \cp_{\textup{rtrct}} :  \vp^* \notin [\vp,\boldsymbol{\infty}) \ \forall \vp \in  \cp_{\textup{rtrct}} \setminus \{\vp^*\} \}, \\
	\mathcal{P}_{\textup{rsopt}} : = \{\vp^* \in \cp_{\textup{rstrct}} :  \vp^* \notin [\vp,\boldsymbol{\infty}) \ \forall \vp \in  \cp_{\textup{rstrct}} \setminus \{\vp^*\} \}.
\end{gather}
\end{definition}

Define $\theUB: \naturals \to \natzero \cup \{\infty\}$
\begin{equation} \label{eq:Mddef}
    \theUB(d) := \begin{cases}
        \min\{n \in \natzero : \lambda_{n+1,d} \le \theM(d)\}, & \theM(d) > 0, \\
        \infty, & \theM(d) = 0.
    \end{cases}
\end{equation}
Also, for each $d \in \naturals$ define the set
\begin{equation} \label{eq:hIdef}
\mathcal{I}_d :=\begin{cases}
\emptyset, & \theUB(d) = 0, \\
\{1,\ldots, \theUB(d)\}, & 0 < \theUB(d) < \infty, \\
\naturals, & \theUB(d) = \infty.
\end{cases}
\end{equation}

From the definition of $\theUB(d)$ it follows that $(\lambda_{i,d}^{-1},d) \in \Omega  \iff i \in \hI_d$, and
\begin{equation} \label{eq:comprestr}
  \comp(\varepsilon,d) \begin{cases}
= 0, & \theUB(d) = 0, \\
    \in \hI_d, & \theUB(d) > 0,
\end{cases}
\qquad \forall (\varepsilon^{-1},d) \in \Omega.
\end{equation}

The equivalent conditions for (strong) $T$-tractability on a restricted domain in Theorems \ref{thm_main_rest_strong_tract} and \ref{thm_main_rest_tract} mimic the conditions given in Theorems \ref{thm_main_strong_tract2} and \ref{thm_main_tract2}, except that the upper limits on the sums now correspond to $\theUB(d)$ rather than $\infty$.  The proofs are a bit more delicate, but similar arguments are used.

We note that it may be possible to define $\Omega$ such that  $\lambda_{1,d} \le \theM(d)$
for all $d \in \naturals$. In this case, $\theUB(d) = 0$ for all $d \in \naturals$, the equivalent conditions for (strong) tractability are trivially satisfied, and the zero algorithm satisfies the error tolerance.

\subsection{Strong Tractability on a Restricted Domain}

\begin{theorem}\label{thm_main_rest_strong_tract}
Let $T$ be a tractability function as specified in \eqref{eq:Tspec} and satisfying \eqref{eq:Tconditions} and \eqref{eq:ptauassume}.  A problem is strongly $T$-tractable \emph{on the restricted domain $\Omega$} iff there exists $\vp \in [\vzero, \boldsymbol{\infty})$ and an integer $L_{\vp} > 0$ such that
\begin{equation} \label{eq:strong_rest_tractiff}
     S_{\vp}:=\sup_{d \in \naturals} \sum_{i = L_{\vp}}^{\theUB(d)} \frac{1}{T(\lambda_{i,d}^{-1},1,\vp)} < \infty.
\end{equation}
By convention, if $\theUB(d) < L_{\vp}$, the sum is zero.

If \eqref{eq:strong_rest_tractiff} holds for some $\vp$, let  $\widetilde{\cp}_{\textup{rstrct}} : = \{\vp^* : \eqref{eq:strong_rest_tractiff} \text{ holds }\forall \vp \in (\vp^*,\boldsymbol{\infty})\}$.  Then $\cp_{\textup{rstrct}} = \widetilde{\cp}_{\textup{rstrct}}$, and the set of optimal strong tractability parameters on the restricted domain is
\[
	\cp_{\textup{rsopt}} =
	\{\vp^* \in \widetilde{\cp}_{\textup{rstrct}} :  \vp^* \notin [\widetilde{\vp},\boldsymbol{\infty}) \ \forall \widetilde{\vp}\in  \widetilde{\cp}_{\textup{rstrct}} \setminus \{\vp^*\} \}.
\]
\end{theorem}

\begin{proof}
\textbf{Sufficient condition:}\newline
Fix $d$.  If $\theUB(d) = 0$, then $\comp(\varepsilon,d) = 0 \le C_{\vp}\, T(\varepsilon^{-1},1,\vp)$ automatically for all $\varepsilon \in [\theM(d), \infty)$.

For the case $\theUB(d)> 0$, following the argument for the sufficient condition in Theorem \ref{thm_main_strong_tract2} it follows that for all $\thed \in \naturals$ and integers $n, N \in \hI_d$  with $N \le n$, we have
\begin{align*}
    \lambda_{n+1,d} \le \lambda_{n,d} \le \cdots \le \lambda_{N,d}
    & \implies \frac{1}{T(\lambda_{n+1, d}^{-1},\thed,\vp)} \le \frac{1}{T(\lambda_{n, d}^{-1},\thed,\vp)} \le \cdots \le \frac{1}{T(\lambda_{N, d}^{-1},\thed,\vp)} \\
     \implies \frac{1}{T(\lambda_{n+1, d}^{-1},\thed,\vp) }
    \le &\frac{1}{n-N+1} \sum_{i=N}^n  \frac{1}{T(\lambda_{i, d}^{-1},\thed,\vp) }
    \le \frac{1}{n-N+1} \sum_{i=N}^{\theUB(d)}  \frac{1}{T(\lambda_{i, d}^{-1},\thed,\vp)}.
\end{align*}
Thus, we can conclude from the previous line that
\begin{multline} \label{eq:implicate_a_rest}
    n - N +1 \ge T(\varepsilon^{-1},\thed,\vp) \sum_{i=N}^{\theUB(d)} \frac{1}{T(\lambda_{i, d}^{-1},\thed,\vp)} \\
   \implies   \frac{1}{T(\lambda_{n+1, d}^{-1},\thed,\vp)} \le
   \frac{1}{n-N+1} \sum_{i=N}^{\theUB(d)} \frac{1}{T(\lambda_{i, d}^{-1},\thed,\vp) } \le \frac{1}{T(\varepsilon^{-1},\thed,\vp)} \\ \forall \thed \in \naturals,\  N \in \hI_d, \ \mbox{and}\ n\ge N
\end{multline}

Moreover, given that $T$ is increasing in its arguments, we have the following equivalent expression for the information complexity via
\eqref{eq:comp_def}:
\begin{align} 
	\nonumber
	\comp(\varepsilon, d) & = \min \{n \in \hI_d : \lambda_{n+1, d} \le \varepsilon\} \qquad \text{by \eqref{eq:comprestr}} \\
	\nonumber
	& = \min \biggl\{n \in \hI_d : \frac{1}{T(\lambda_{n+1,d}^{-1},\thed,\vp)}\le \frac{1}{T(\varepsilon^{-1},\thed,\vp) }\biggr\} \qquad  \forall \thed\in \naturals\\
	\nonumber
	& \le  \min \biggl\{n \in \hI_d : n \ge N - 1 + T(\varepsilon^{-1},\thed,\vp) \sum_{i=N}^{\theUB(d)} \frac{1}{T(\lambda_{i, d}^{-1},\thed,\vp)} \biggr\} \;  \forall  \thed, \ 1 \le N \le \theUB(d) \\
 \nonumber
 & \hspace{20ex}\text{by \eqref{eq:implicate_a_rest}.}
 \intertext{
 This last statement holds for all $N \le \theUB(d)$ because $n \ge N$ automatically in the expression above.  Thus,}
 \comp(\varepsilon,d)
	& \le  T(\varepsilon^{-1},\thed,\vp) \Biggl [\frac{N}{T(\varepsilon^{-1},\thed,\vp)}  +  \sum_{i=N}^{\theUB(d)} \frac{1}{T(\lambda_{i, d}^{-1},\thed,\vp)} \Biggr] \quad \forall  \thed,\ 1 \le N \le \theUB(d). \label{eq:compUP_rest}
\end{align}

Now we take this upper bound further, specializing to the case of $\thed=1$ and $N = \min\{L_{\vp},\theUB(d)\}$:
\begin{align*}
       \comp(\varepsilon,d)
       & \le T(\varepsilon^{-1},1, \vp) \Biggl [ \frac{\min\{L_{\vp},\theUB(d)\}}{T(\varepsilon^{-1},1, \vp)}  + \underbrace{\sum_{i = \min\{L_{\vp},\theUB(d)\}}^{\theUB(d)} \frac{1}{T(\lambda_{i, d}^{-1},1,\vp)}}_{\le S_{\vp} + 1/ T(\lambda_{\theUB(d), d}^{-1},1,\vp) \text{ by \eqref{eq:strong_rest_tractiff}} }
        \Biggr] \\
       & \le T(\varepsilon^{-1},1,\vp) \underbrace{\left[\frac{L_{\vp}+1}{T(0,1, \vp)} + S_{\vp} \right]}_{=:C_{\vp}}
       \qquad \text{by \eqref{eq:Tnondecreasing}}\\
       & =  C_{\vp}T(\varepsilon^{-1},1,\vp).
\end{align*}
This means that we have strong $T$-tractability on the restricted domain $\Omega$ via \eqref{eq:rest_strong_tract_def}, and verifies the sufficiency of \eqref{eq:strong_rest_tractiff}.

\bigskip
\noindent \textbf{Necessary condition:} \\
Suppose that we have strong
$T$-tractability on the restricted domain $\Omega$ as defined in \eqref{eq:rest_strong_tract_def}. That is, for some $\vp \ge \vzero$ there exists a positive constant $C_{\vp}$ such that
\[
\comp(\varepsilon,d)\le C_{\vp}\, T(\varepsilon^{-1},1,\vp)
\qquad \forall (\varepsilon^{-1}, d) \in \Omega.
\]
Since the sequence of singular values $\lambda_{1,d}, \lambda_{2,d}, \ldots $ is non-increasing, we have
\begin{equation}\label{eq:lambda_K_strong4}
\lambda_{\lfloor C_{\vp}\, T(\varepsilon^{-1},1,\vp)\rfloor +1,d}\le \varepsilon \qquad\forall (\varepsilon^{-1}, d) \in \Omega.
\end{equation}

For  all $\varepsilon > 0$, define the positive integers
\[
i(\varepsilon,\vp):= \lfloor C_{\vp}\, T(\varepsilon^{-1},1,\vp)\rfloor +1, \quad
i(\infty,\vp) =  \lfloor C_{\vp}\, T(0,1,\vp)\rfloor +1 \ge 1,
\]
just as in \eqref{eq:iepsp}.
Thus, it follows by \eqref{eq:lambda_K_strong4} that
\begin{equation} \label{eq:lambdaiep}
\lambda_{i(\varepsilon,\vp),d} \le \varepsilon \qquad
 \forall (\varepsilon^{-1}, d) \in \Omega.
\end{equation}
Note furthermore that we always have
\[
i(\varepsilon,\vp)\le C_{\vp}\, T(\varepsilon^{-1},1,\vp)+1 \le C_{\vp} T(\lambda_{i(\varepsilon,\vp),d}^{-1},1,\vp)+1 \qquad  \forall (\varepsilon^{-1}, d)\in \Omega,
\]
since
$T(\cdot,1, \vp)$ is non-decreasing.

For $\varepsilon$ taking on all values in $(\theM(d),\infty)$, $i(\varepsilon,\vp)$ takes on (at least) all values in
\[
\mathscr{I}_{d,\vp}:=\begin{cases}
\emptyset, & i(\theM(d),\vp) = i(\infty,\vp), \\
\{i(\infty,\vp),\ldots, i(\theM(d),\vp)-1\}, & \theM(d) > 0, \\
\{i(\infty,\vp), i(\infty,\vp) + 1, \ldots\}, & \theM(d) = 0.
\end{cases}
\]
So,
\[
i\le  C_{\vp} T(\lambda_{i,d}^{-1},1,\vp)+1 \qquad \forall i\in \mathscr{I}_{d,\vp},  \ d \in \naturals.
\]
This implies via our technical assumption \eqref{eq:ptauassume} that, for $\tau>1$,
\begin{gather}
\nonumber
 K_{\vp,\tau}\,T (\lambda_{i,d}^{-1},1,\tau \vp) \ge
 [T(\lambda_{i,d}^{-1},1, \vp)]^\tau
 \ge
  \left[\frac{(i-1)}{C_{\vp}}\right]^\tau \qquad \forall i\in \mathscr{I}_{d,\vp},  \ d \in \naturals, \\
  \label{eq:rest_strong_ineq}
 \frac{1}{T (\lambda_{i,d}^{-1},1,\tau \vp)} \le
\frac{K_{\vp,\tau}\, C_{\vp}^\tau}{(i-1)^\tau}\qquad \forall i \in \mathscr{I}_{d,\vp} \setminus\{1\}, \ d \in \naturals .
\end{gather}

To complete the proof, we will sum both sides of \eqref{eq:rest_strong_ineq} over the range of $i$ appearing in \eqref{eq:strong_rest_tractiff}, to establish that $S_{\tau\vp}$ is finite.  To do this we need to show that \eqref{eq:rest_strong_ineq} holds for that range.  Choose $L_{\tau\vp} = i(\infty,\vp) +1 \ge 2$. There are three cases.
\begin{enumerate}
\renewcommand{\labelenumi}{\roman{enumi})}

\item If $\theUB(d) < L_{\tau\vp}$, then
\[
\sum_{i = L_{\tau\vp}}^{\theUB(d)} \frac{1}{T (\lambda_{i,d}^{-1},1, \tau \vp)} = 0,
\]
which does not affect the finiteness of $S_{\tau\vp}$.

\item If $\theUB(d) \ge L_{\tau\vp}$ and $\theM(d) = 0$, then $\{L_{\tau\vp},  L_{\tau\vp} + 1, \ldots \} \subseteq \mathscr{I}_{d,\vp} \setminus \{1\}$.

\item If $\theUB(d) \ge L_{\tau\vp}$ and $\theM(d) >0$, then \eqref{eq:Mddef} implies that $\lambda_{\theUB(d)+1,d} \le \theM(d) < \lambda_{\theUB(d),d}$.  Moreover, there must be some $\varepsilon^*$ satisfying $\theM(d) < \varepsilon^* < \lambda_{\theUB(d),d}$, for which $\lambda_{i(\varepsilon^*,\vp)} \le \varepsilon^* < \lambda_{\theUB(d),d}$ by \eqref{eq:lambdaiep}.  Because of the ordering of the singular values, this implies that  $i(\varepsilon^*,\vp) \ge \theUB(d)+1$.  Since $i(\cdot,\vp)$ is nonincreasing, $i(\theM(d),\vp) \ge \theUB(d)+1$.
This means that $\theUB(d) \le i(\theM(d),\vp)-1$.  So, $\{L_{\tau\vp}, \ldots, \theUB(d)\} \subseteq \mathscr{I}_{d,\vp} \setminus \{1\}$.

\bigskip

In both the second and third cases,
\begin{align*}
S_{\tau\vp} & = \sup_{d\in\naturals} \sum_{i = L_{\tau\vp}}^{\theUB(d)} \frac{1}{T (\lambda_{i,d}^{-1},1, \tau \vp)}
\le
\sup_{d\in\naturals} \sum_{i \in \mathscr{I}_{d,\vp} \setminus\{1\} } \frac{1}{T (\lambda_{i,d}^{-1},1, \tau \vp)} \\
&  \le  K_{\vp,\tau}\, C_{\vp}^\tau
\sum_{i=2}^\infty \frac{1}{(i-1)^\tau}
 \le K_{\vp,\tau}\, C_{\vp}^\tau
\zeta (\tau)
 < \infty \qquad \text{by \eqref{eq:rest_strong_ineq},}
\end{align*}
where $\zeta$ denotes the Riemann zeta function.

\end{enumerate}
Thus, in all three cases, the assumption of strong tractability on the restricted domain implies \eqref{eq:strong_rest_tractiff} with $\vp$ replaced by $\vp' = \tau \vp$, and so we see the necessity of $\eqref{eq:strong_rest_tractiff}$.

\bigskip
\noindent \textbf{Optimality:} \\
The proof of optimality is analogous to that for Theorem \ref{thm_main_strong_tract2} and is omitted.

\bigskip

\noindent This concludes the proof of  Theorem \ref{thm_main_rest_strong_tract}.  
\end{proof}

\subsection{Tractability on a Restricted Domain}

The equivalent condition for tractability on a restricted domain is analogous to the conditions for tractability on an unrestricted domain and for strong tractability on a restricted domain.  The proof is analogous as well.

\begin{theorem}\label{thm_main_rest_tract}
Let $T$ be a tractability function as specified in \eqref{eq:Tspec} and satisfying \eqref{eq:Tconditions} and \eqref{eq:ptauassume}.  A problem is $T$-tractable on the restricted domain $\Omega$ iff there exists $\vp\ge \vzero$ and a positive constant $L_{\vp}$ such that
\begin{equation} \label{eq:rest_tractiff}
     S_{\vp}:=\sup_{d \in \naturals}
     \sum_{i = \lceil L_{\vp}\, T(0,d,\vp) \rceil}^{\theUB(d)} \frac{1}{T(\lambda^{-1}_{i,d},d,\vp)}< \infty.
\end{equation}
If \eqref{eq:rest_tractiff} holds for some $\vp$, let $\widetilde{\cp}_{\textup{rtrct}} : = \{\vp^* : \eqref{eq:rest_tractiff} \text{ holds }\forall \vp \in (\vp^*,\boldsymbol{\infty}) \}$.
Then $\cp_{\textup{rtrct}} = \widetilde{\cp}_{\textup{rtrct}}$, and the set of optimal  tractability parameters is
\[
\cp_{\textup{ropt}} =
\{\vp^* \in \widetilde{\cp}_{\textup{rtrct}} :  \vp^* \notin [\widetilde{\vp},\boldsymbol{\infty}) \ \forall \widetilde{\vp}\in  \widetilde{\cp}_{\textup{rtrct}} \setminus \{\vp^*\} \}.
\]
\end{theorem}

\begin{proof}
    \textbf{Sufficient condition:}\\
Fix $d$.  If $\theUB(d) = 0$, then $\comp(\varepsilon,d) = 0 \le C_{\vp}\, T(\varepsilon^{-1},d,\vp)$ automatically for all $\varepsilon \in [\theM(d), \infty)$.

For the case $\theUB(d)> 0$, we utilize the upper bound \eqref{eq:compUP_rest}, specializing to the case of $\thed=d$ and $N = \min\left\{\lceil L_{\vp}\, T(0,d,\vp)\rceil,\theUB(d)\right\}$:
\begin{align*}
       \comp(\varepsilon,d)
       & \le T(\varepsilon^{-1},d, \vp) \Biggl [ \frac{\min\left\{\lceil L_{\vp}\, T(0,d,\vp)\rceil,\theUB(d)\right\}}{T(\varepsilon^{-1},d, \vp)}  + \underbrace{\sum_{i = \min\left\{\lceil L_{\vp}\, T(0,d,\vp)\rceil,\theUB(d)\right\}}^{\theUB(d)} \frac{1}{T(\lambda_{i, d}^{-1},d,\vp)}}_{\le S_{\vp} + 1/ T(\lambda_{\theUB(d), d}^{-1},d,\vp) \text{ by \eqref{eq:rest_tractiff}}} 
        \Biggr] \\
       & \le T(\varepsilon^{-1},d,\vp)\left[\frac{ L_{\vp}\, T(0,d,\vp)+1}{T(0,d, \vp)} + S_{\vp} + \frac{1}{T(\lambda_{\theUB(d), d}^{-1},d,\vp)} \right]
       \qquad \text{by \eqref{eq:Tnondecreasing}}\\
       & \le T(\varepsilon^{-1},d,\vp) \underbrace{\left[L_{\vp} + \frac{2}{T(0,1, \vp)} + S_{\vp} \right]}_{=:C_{\vp}}
       \qquad \text{by \eqref{eq:Tnondecreasing}}\\
       & =  C_{\vp}T(\varepsilon^{-1},d,\vp).
\end{align*}
This means that we have strong $T$-tractability on the restricted domain $\Omega$ via \eqref{eq:rest_tract_def}, and verifies the sufficiency of \eqref{eq:rest_tractiff}.

\bigskip

\noindent \textbf{Necessary condition:}\\
Suppose that we have
$T$-tractability on the restricted domain. That is, for some $\vp \ge \vzero$, there exists a positive constant $C_{\vp}$ such that
\[
\comp(\varepsilon,d)\le C_{\vp}\, T(\varepsilon^{-1},d,\vp) \qquad \forall (\varepsilon^{-1}, d) \in \Omega.
\]
Since the sequence of singular values $\lambda_{1,d}, \lambda_{2,d}, \ldots $ is non-increasing, we have
\begin{equation}\label{eq:rest_lambda_K4}
	\lambda_{\lfloor C_{\vp}\, T(\varepsilon^{-1},d,\vp)\rfloor +1,d}\le \varepsilon \qquad
 \forall (\varepsilon^{-1}, d) \in \Omega.
\end{equation}

For $(\varepsilon^{-1}, d) \in \Omega$, define the positive integers
\[
i (\varepsilon,d,\vp):= \lfloor C_{\vp}\, T(\varepsilon^{-1},d,\vp)\rfloor +1, \qquad
i (\infty,d,\vp) := \lfloor C_{\vp}\, T(0,d,\vp)\rfloor +1 \ge 1.
\]
It follows by \eqref{eq:rest_lambda_K4} that $\lambda_{i(\varepsilon,d,\vp),d} \le \varepsilon$
for any  $(\varepsilon^{-1}, d) \in \Omega$.
Note furthermore that since $T(\cdot,d,\vp)$ is non-decreasing, we have
\[
i(\varepsilon,d,\vp)\le C_{\vp}\, T(\varepsilon^{-1},d,\vp)+1 \le C_{\vp} T(\lambda_{i(\varepsilon,d,\vp),d}^{-1},d,\vp)+1 \qquad \forall (\varepsilon^{-1}, d) \in \Omega.
\]
For $\varepsilon$ taking on all values in $(\theM(d),\infty)$, $i(\varepsilon,d,\vp)$ takes on (at least) all values in
\[
\mathscr{I}_{d,\vp}:=\begin{cases}
\emptyset, & i(\theM(d),d,\vp) = i(\infty,d,\vp), \\
\{i(\infty,d,\vp),\ldots, i(\theM(d),d,\vp)-1\}, & \theM(d) > 0, \\
\{i(\infty,d,\vp), i(\infty,d,\vp) + 1, \ldots\}, & \theM(d) = 0.
\end{cases}
\]
So,
\[
i\le  C_{\vp} T(\lambda_{i,d}^{-1},d,\vp)+1, \qquad \forall i\in \mathscr{I}_{d},  \ d \in \naturals.
\]
This implies via our technical assumption \eqref{eq:ptauassume} that, for $\tau>1$,
\begin{gather}
\nonumber
 K_{\vp,\tau}\,T (\lambda_{i,d}^{-1},d,\tau \vp) \ge
 [T(\lambda_{i,d}^{-1},d, \vp)]^\tau
 \ge
  \left[\frac{(i-1)}{C_{\vp}}\right]^\tau \qquad \forall i\in \mathscr{I}_{d,\vp},  \ d \in \naturals,\\
  \label{eq:rest_ineq}
 \frac{1}{T (\lambda_{i,d}^{-1},d,\tau \vp)} \le
\frac{K_{\vp,\tau}\, C_{\vp}^\tau}{(i-1)^\tau}\qquad \forall i \in \mathscr{I}_{d,\vp} \setminus\{1\}, \ d \in \naturals .
\end{gather}

To complete the proof, we will sum both sides of \eqref{eq:rest_ineq} over the range of $i$ appearing in \eqref{eq:rest_tractiff}, to establish that $S_{\vp'}$ is finite, where  $\vp' = \tau \vp$.  To do this we need to show that \eqref{eq:rest_ineq} holds for that range.

For any $d\in\naturals$, let
\begin{align*}
L_{\vp'} & := C_{\vp'} + \frac{2}{T(0,1,\vp')},
\intertext{and note that }
 \max\{i(\infty,d,\vp'),2\} & \le i(\infty,d,\vp') +1  = \lfloor C_{\vp'}\, T(0,d,\vp')\rfloor +2 \le  \left[C_{\vp'} + \frac{2}{T(0,d,\vp')}   \right]\, T(0,d,\vp')\\
 & =  L_{\vp'}\, T(0,d,\vp')
 \le \lceil L_{\vp'} T(0,d,\vp')\rceil .
\end{align*}
This is the inequality needed for the lower limit of the sum.

There are three cases.
\begin{enumerate}
\renewcommand{\labelenumi}{\roman{enumi})}

\item If $\theUB(d) <\lceil L_{\vp'} T(0,d,\vp')\rceil$, then
\[
\sum_{i = \lceil L_{\vp'} T(0,d,\vp')\rceil }^{\theUB(d)} \frac{1}{T (\lambda_{i,d}^{-1},d, \vp')} = 0,
\]
which does not affect the finiteness of $S_{\vp'}$.

\item If $\theUB(d) \ge \lceil L_{\vp'} T(0,d,\vp')\rceil$ and $\theM(d) = 0$, then $\{\lceil L_{\vp'} T(0,d,\vp')\rceil,  \lceil L_{\vp'} T(0,d,\vp')\rceil + 1, \ldots \} \subseteq \mathscr{I}_{d,\vp} \setminus \{1\}$.

\item If $\theUB(d) \ge \lceil L_{\vp'} T(0,d,\vp')\rceil$ and $\theM(d) >0$, then \eqref{eq:Mddef} and \eqref{eq:rest_lambda_K4} imply that
\[
\lambda_{i(\theM(d),d,\vp),d} \le \lambda_{\theUB(d)+1,d} \le \theM(d) < \lambda_{\theUB(d),d}.
\]
By the same argument as in the proof of strong tractability, $\theUB(d) \le i(\theM(d),d,\vp)-1$.  So, $\{ \lceil L_{\vp'} T(0,d,\vp')\rceil, \ldots, \theUB(d)\} \subseteq \mathscr{I}_{d,\vp} \setminus \{1\}$.

\bigskip

In both the second and third cases,
\begin{align*}
S_{\vp'} & = \sup_{d\in\naturals} \sum_{i =  \lceil L_{\vp'} T(0,d,\vp')\rceil}^{\theUB(d)} \frac{1}{T (\lambda_{i,d}^{-1},d, \tau \vp)}
\le
\sup_{d\in\naturals} \sum_{i \in \mathscr{I}_{d,\vp} \setminus\{1\} } \frac{1}{T (\lambda_{i,d}^{-1},d, \tau \vp)} \\
&  \le  K_{\vp,\tau}\, C_{\vp}^\tau
\sum_{i=2}^\infty \frac{1}{(i-1)^\tau}
 \le K_{\vp,\tau}\, C_{\vp}^\tau
\zeta (\tau)
 < \infty \qquad \text{by \eqref{eq:rest_ineq}},
\end{align*}
where $\zeta$ denotes the Riemann zeta function.

\end{enumerate}
Thus, in all three cases, the assumption of strong tractability on the restricted domain implies \eqref{eq:rest_tractiff} with $\vp$ replaced by $\vp' = \tau \vp$, and so we see the necessity of $\eqref{eq:rest_tractiff}$.
\bigskip

\bigskip

\noindent \textbf{Optimality:} \\
The proof of optimality is analogous to that for Theorem \ref{thm_main_tract2} and is omitted.

\bigskip

\noindent
This concludes the proof of Theorem \ref{thm_main_rest_tract}.  
\end{proof}

\section{Discussion and Further Work}

We have shown that many proofs of equivalent conditions for (strong) tractability of various types can be unified into a handful of proofs.  Even the arguments underlying these handful are quite similar and could be consolidated even further.  This unification simplifies our understanding of these equivalent conditions and spotlights the key ideas about what makes a class of problems hard.

However, some questions remain:
\begin{itemize}
    \item What interesting cases of $T$ do not correspond to known tractability measures? 
    \item What interesting cases are there where there is no (strong) tractability on an unrestricted domain, but there is (strong) tractability on a nontrivial restricted domain, such as $(0,\infty) \times \{1, \ldots, d_{\max}] \cup (\varepsilon_{\min}^{-1},\infty) \times \naturals$?
     \item For Hilbert spaces and solution operators, $\{\SOL_d : \cf_d \to \cg_d\}_{d \in \naturals}$, of tensor product form, the singular values take the form $\lambda_{i,d} = \lambda_{i,d,1} \cdots  \lambda_{i,d,d}$.  What more can be said about the  equivalent conditions  for (strong) tractability in this case?
    \item For information complexity defined over cones of functions rather than balls (see e.g., \cite{DinHic20a,DinEtal20a}), what can be said about equivalent conditions for (strong) tractability?
\end{itemize}

\section{Acknowledgments}
The authors gratefully acknowledge the fruitful discussions with Henryk Wo\'{z}niakowski and Erich Novak. Moreover, we thank two anonymous referees for their suggestions.  Fred Hickernell acknowledges support by the United States National Science Foundation under Award No.\ 2316011.
Peter Kritzer was supported by the Austrian Science Fund (FWF) Project P34808. For the purpose of open access, the authors have applied a CC BY public copyright licence to any author accepted manuscript version arising from this submission.

\section*{Appendix 1: A concrete example}

We define the weighted Korobov space $\ch_{\kor,d,\bsa,\bsb}$ below. 
This space contains real-analytic functions, and has been studied in the context of exponential tractability in, e.g., \cite{DKPW14, KPW17}, where the reader can find further details on this space.

Let $\bsa=\{a_j\}_{j \ge 1}$ and
$\bsb=\{b_j\}_{j \ge 1}$ be two non-decreasing sequences of real positive weights.
Fix $\omega\in(0,1)$, and write 
\[
r_{d,\bsa,\bsb} (\bsh):=\omega^{-\sum_{j=1}^{d}a_j \abs{h_j}^{b_j}}
\qquad\mbox{for}\qquad \bsh=(h_1,h_2,\dots,h_d)\in\ZZ^d.
\]

We define $\ch_{\kor,d,\bsa,\bsb}$ as the space of all one-periodic functions $f$ with absolutely convergent Fourier series given by $f(\bsx) = \sum_{\bsh \in \ZZ^d} \widehat{f}(\bsh) \exp(2 \pi \sqrt{-1}\bsh'\bsx)$, 
and with finite norm 
$||f||_{\kor,d,\bsa,\bsb}:=\langle f, f\rangle_{\kor,d,\bsa,\bsb}^{1/2}$,
where the inner product is given by
\[
 \langle f, g\rangle_{\kor,d,\bsa,\bsb}:=
 \sum_{\bsh\in\ZZ^d} r_{d,\bsa,\bsb} (\bsh) 
 \widehat{f} (\bsh) \overline{\widehat{g}(\bsh)}.
\]
Let us consider $L_2$-approximation, i.e., $\SOL_d:\ch_{\kor,d,\bsa,\bsb}\to L_2([0,1]^d)$, $\SOL_d(f)=f$, and let us allow arbitrary continuous linear functionals as information. One can show that the singular values  for this choice of $\SOL_d$ are given by
\begin{equation}\label{eq:eigenvalues}
 \left\{\lambda_{n,d}\colon n\in\NN\right\}=\left\{\left(r_{d,\bsa,\bsb} (\bsh)\right)^{-1/2}\colon \bsh\in\ZZ^d\right\}
 =\left\{\omega^{\frac{1}{2}\sum_{j=1}^{d}a_j \abs{h_j}^{b_j}}\colon \bsh\in\ZZ^d\right\}.
\end{equation}

If, for instance, we would like to find out whether $L_2$-approximation on $\ch_{\kor,d,\bsa,\bsb}$ satisfies 
exponential strong tractability, we would use condition \eqref{eq:KW_exp_spt}, which previously existed in the literature. A slight drawback of this condition, however, is that we need to know everything about the \emph{order} of the singular values, since we need to study summability of $\lambda_{i,d}^{i^{-\tau}}$ in \eqref{eq:KW_exp_spt}. This is technically rather involved, and was implicitly done in \cite{DKPW14} (we remark that the concise form 
of the condition \eqref{eq:KW_exp_spt} was not yet known when \cite{DKPW14} was written, and the proof idea in \cite{DKPW14} is less 
straightforward than just checking \eqref{eq:KW_exp_spt}.

Alternatively, we can use the newly derived condition \eqref{eq:exp_spt} instead of \eqref{eq:KW_exp_spt} and study the expression
\[
 \sum_{i=1}^\infty \frac{1}{T\left(\lambda_{i,d}^{-1},1,(p,0)\right)},
\]
for a real $p>0$. Inserting the singular values as in \eqref{eq:eigenvalues}, yields
\begin{eqnarray*}
  \sum_{i=1}^\infty \frac{1}{T\left(\lambda_{i,d}^{-1},1,(p,0)\right)}
  &=&
\sum_{\bsh\in\ZZ^d}
\frac{1}{\left[\max\left\{1,\log\left(1+\omega^{-\frac{1}{2}\sum_{j=1}^d a_j \abs{h_j}^{b_j}}\right)
\right\}\right]^p}\\
&\le&1 +\sum_{\bsh\in\ZZ^d\setminus \{\boldsymbol{0}\}} \frac{1}{\left(\log \left(\omega^{-\frac{1}{2}\sum_{j=1}^d a_j \abs{h_j}^{b_j}}\right)\right)^p} \, .
\end{eqnarray*}
It is then not hard to show that the two conditions 
\[
 \liminf_{j\to\infty}\frac{\log a_j}{j}>0 \quad \mbox{and}\quad \sum_{j=1}^\infty \frac{1}{b_j} <\infty 
\]
are sufficient to achieve exponential strong tractability (for details we refer to \cite{K24}), and this is in accordance with the findings in \cite{DKPW14}. Note, however, that this result in \cite{DKPW14} was technically much harder to show than 
the analysis of the previous sum is.

\section*{Appendix 2: Proof of Proposition~\ref{prop:equiv_conditions}}

\begin{proof}[Proposition 1]
We will give the proof of equivalence of the conditions \eqref{eq:exp_pt} and \eqref{eq:KW_exp_pt}. The proof of the equivalence of \eqref{eq:exp_spt} and \eqref{eq:KW_exp_spt} is analogous and simpler (actually, the latter proof can be seen as a special case of the former, by assuming all exponents of $d$ to be zero).

Let us first assume that \eqref{eq:KW_exp_pt} holds, i.e.,
\[
M:=\sup_{d\in\naturals} d^{-\tau_1} \sum_{i=\lceil C d^{\tau_3} \rceil}^{\infty} \lambda_{i,d}^{i^{-\tau_2}} < \infty
\]
for some $\tau_1,\tau_3\ge 0$ and $\tau_2, C>0$. Here, $M$ and other constants above may depend on $\tau_1$, $\tau_2$, and $\tau_3$.
Define $B_d$ as the set of indices that are contained in the above sum and for which the terms in the sum are relatively large:
\[
B_d:=\left\{i\in\naturals \colon i\ge \left\lceil Cd^{\tau_3}\right\rceil\quad\mbox{and}\quad
\lambda_{i,d}^{i^{-\tau_2}}>\frac{1}{e}\right\}.
\]
Due to our assumption, we see that $\card(B_d) < e\,M\,d^{\tau_1}$ and $\card(B_d)\le \lfloor e\,M\,d^{\tau_1} \rfloor$. 
Suppose now that $i\ge \lceil C d^{\tau_3}\rceil$ but $i\not\in B_d$, which means that the term in the above sum is relatively small:
\begin{equation}
    \lambda_{i,d}^{i^{-\tau_2}} \le \frac{1}{e} \iff \log (\lambda_{i,d}^{-1})\ge i^{\tau_2}.
\label{eq:logandi}
\end{equation}

Now, choose $p> 1/\tau_2$, $q=\max\{\tau_1,\tau_3\}$, and $L_{(p,q)}=\max\{C, e\, M\}$, which implies that $\lceil L_{(p,q)} d^{q}\rceil \ge e\, M\, d^{\tau_1}$. Then, for any $d\in\naturals$, the desired sum in \eqref{eq:exp_pt} can be bounded by splitting it into the sum of the larger terms and the sum of the smaller terms:
\begin{align*}
\MoveEqLeft{d^{-q} \sum_{i=\lceil L_{(p,q)} d^{q}\rceil}^\infty \frac{1}{[\max\{1,\log(1+\lambda_{i,d}^{-1})\}]^p}} \\
&=d^{-q} \sum_{\substack{i=\lceil L_{(p,q)} d^{q}\rceil\\ i\in B_d}}^\infty \frac{1}{[\max\{1,\log(1+\lambda_{i,d}^{-1})\}]^p} + d^{-q} \sum_{\substack{i=\lceil L_{(p,q)} d^{q}\rceil\\ i\not\in B_d}}^\infty \frac{1}{[\max\{1,\log(1+\lambda_{i,d}^{-1})\}]^p}\\
&\le d^{-q} e\, M\, d^{\tau_1} + d^{-q}\sum_{\substack{i=\lceil L_{(p,q)} d^{q}\rceil\\ i\not\in B_d}}^\infty \frac{1}{[\log (\lambda_{i,d}^{-1})]^p}\\
&\le e\, M + d^{-q}\sum_{\substack{i=\lceil L_{(p,q)} d^{q}\rceil\\ i\not\in B_d}}^\infty \frac{1}{i^{\tau_2 p}} \qquad \text{by \eqref{eq:logandi}}\\
&\le  e\, M + \zeta (\tau_2 p).
\end{align*}
Since $p$ was chosen to be strictly larger than $1/\tau_2$, it follows that the latter expression is finite and its value is independent of $d$. Therefore, Condition \eqref{eq:exp_pt} holds.

Conversely, let us now assume that Condition \eqref{eq:exp_pt} holds for some $p> 0$ and $q\ge 0$ that are assumed to be fixed. All constants in this part of the proof may depend on $p$ and $q$.  We will show next that this implies Condition \eqref{eq:KW_exp_pt}. Indeed, due to \eqref{eq:exp_pt}, there exist constants $M > 1$ and $L_{(p,q)}>0$ such that for all $d\in\naturals$,
\[
d^{-q}
\sum_{i=\lceil L_{(p,q)} d^{q}\rceil}^\infty \frac{1}{[\max\{1,\log(1+\lambda_{i,d}^{-1})\}]^p} \le M,
\]
which is equivalent to
\[
\sum_{i=\lceil L_{(p,q)} d^{q}\rceil}^\infty \frac{1}{[\max\{1,\log(1+\lambda_{i,d}^{-1})\}]^p} \le d^q M.
\]
Similarly to the proof of the sufficient condition in Theorem \ref{thm_main_strong_tract2}, we see that for any $n > \lceil L_{(p,q)} d^{q}\rceil $,
\[
\frac{1}{T(\lambda_{n+1,d}^{-1},1,\vp)} \le \frac{1}{n-\lceil L_{(p,q)} d^{q}\rceil +1 }
\sum_{i=\lceil L_{(p,q)} d^{q}\rceil}^\infty  \frac{1}{T(\lambda_{i,d}^{-1},1,\vp)}
\le \frac{d^q M}{n-\lceil L_{(p,q)} d^{q}\rceil +1 }.
\]
This inequality implies for our specific choice of $T$ that
\[
\frac{n}{d^q M}-\frac{L_{(p,q)}}{M}\le\frac{n-\lceil L_{(p,q)} d^{q}\rceil +1 }{d^q M} \le [\max\{1,\log(1+\lambda_{n+1,d}^{-1})\}]^p
\le [1 + \log(1+\lambda_{n+1,d}^{-1})]^p.
\]
Since both $L_{(p,q)}$ and $M$ in the latter chain of inequalities are independent of $d$ and $n$, there
exists a positive constant $K$ such that
\[
\frac{n}{d^q K} \le [1 + \log(1+\lambda_{n+1,d}^{-1})]^p\quad \forall n > \lceil L_{(p,q)} d^{q}\rceil,
\]
which implies
\[
\frac{n^{1/p}}{d^{q/p} K^{1/p}} -1 \le \log(1+\lambda_{n+1,d}^{-1})\quad \forall n > \lceil L_{(p,q)} d^{q}\rceil.
\]
This, in turn, implies the existence of another constant $\widetilde{K}>0$ such that
\[
\lambda_{n,d} \le \exp\left(-\frac{n^{1/p}}{\widetilde{K}^{1/p} d^{q/p}}\right)
\quad \forall n > \lceil L_{(p,q)} d^{q}\rceil-1,
\]
which is essentially the situation described in \cite[p.~118]{KriWoz19a}. We can then choose $\tau_2\in (0,1/p)$,
and proceed as in \cite{KriWoz19a} to obtain
\[
\sum_{i=\lceil L_{(p,q)} d^{q}\rceil}^\infty \lambda_{i,d}^{i^{-\tau_2}}=\mathcal{O}(d^{q/(1-\tau_2 p)}),
\]
with the implied factor in the $\mathcal{O}$-notation independent of $d$. Then, we can choose $\tau_1=q/(1-\tau_2 p)$ and $\tau_3=q$ and see that \eqref{eq:KW_exp_pt} holds.

This concludes the proof of the equivalence of \eqref{eq:exp_pt} and \eqref{eq:KW_exp_pt}. 

\end{proof}

\end{document}